\documentclass[12pt]{amsart}

\usepackage{hyperref}
\usepackage{amssymb, latexsym, amsthm, enumerate, mathptmx}

\usepackage{tikz}
\usetikzlibrary{arrows}
\usepackage[all]{xy}
\usepackage{color}
\usepackage{cleveref}
\usepackage{mathtools}
\usepackage{paralist}
\usepackage{comment}

\newtheorem{thm}{Theorem}[section]
\newtheorem{lem}[thm]{Lemma}
\newtheorem{cor}[thm]{Corollary}
\newtheorem{prop}[thm]{Proposition}
\newtheorem{conj}[thm]{Conjecture}

\theoremstyle{definition} 
\newtheorem{defn}[thm]{Definition} 
\newtheorem{remark}[thm]{Remark}
\newtheorem{example}[thm]{Example}

\DeclareMathOperator{\Sym}{Sym}
\DeclareMathOperator{\Inc}{Inc}
\DeclareMathOperator{\Tor}{Tor}
\DeclareMathOperator{\reg}{reg}
\DeclareMathOperator{\ind}{ind}
\DeclareMathOperator{\ini}{in}
\DeclareMathOperator{\gin}{gin}
\DeclareMathOperator{\rev}{rev}

\def\Icc{{\mathcal I}}
\def\Jcc{{\mathcal J}}

\newcommand{\N}{\mathbb{N}}
\newcommand{\Z}{\mathbb{Z}}
\newcommand{\be}{{\bf e}}

\title{Castelnuovo-Mumford regularity up to symmetry}

\author{Dinh Van Le}
\address{Institut f\"ur Mathematik, Universit\"at Osnabr\"uck, 49069 Osnabr\"uck, Germany}
\email{dlevan@uos.de}

\author{Uwe Nagel}
\address{Department of Mathematics, University of Kentucky, 715 Patterson office tower, Lexington, KY 40506-0027, USA}
\email{uwe.nagel@uky.edu}

\author{Hop D. Nguyen}
\address{Institute of Mathematics, Vietnam Academy of Science and Technology, 18 Hoang Quoc Viet, 10307 Hanoi, Vietnam}
\email{ngdhop@gmail.com}

\author{Tim R\"{o}mer}
\address{Institut f\"ur Mathematik, Universit\"at Osnabr\"uck, 49069 Osnabr\"uck, Germany}
\email{troemer@uos.de}

\begin{document}

\begin{abstract}
We study the asymptotic behavior of the Castelnuovo-Mumford regularity along chains of graded ideals in increasingly larger polynomial rings that are invariant under the action of symmetric groups. A linear upper bound for the regularity of such ideals is established. We conjecture that their regularity grows eventually precisely linearly. We establish this conjecture in several cases, most notably when the ideals are Artinian or squarefree monomial.
\end{abstract}


\maketitle

\section{Introduction}

Chains of ideals in increasingly larger polynomial rings possessing certain symmetries arise naturally in various areas of mathematics, including
algebraic chemistry \cite{AH07, Dr10},
algebraic statistics and toric algebra \cite{BD11,Dr14,DEKL15,DK14,HM13,HS07,SS03},
and group theory \cite{Co67}.

A newly established technique to deal with such a chain is to pass to a non-Noetherian limit of the chain. Typically, this leads to the study of ideals in a polynomial ring $K[X]=K[x_{k,j}\mid 1\le k\le c, j\ge 1]$ ($c\in \N$) over a field $K$ that are invariant under the action of some large monoid, such as $\Sym(\infty)=\bigcup_{n\ge1}\Sym(n)$, where $\Sym(n)$ is the symmetric group on $\{1,\dots,n\}$. The action is induced by $\sigma \cdot x_{k,j}=x_{k,\sigma(j)}$ for $\sigma\in\Sym(\infty)$. With this action any $\Sym(\infty)$-invariant ideal $I$ can also be described by a $\Sym(\infty)$-invariant chain of ideals
\[
I_1\subseteq I_2\subseteq\cdots\subseteq I_n\subseteq\cdots,
\]
in which each truncation $I_n=I\cap K[X_n]$ is an ideal in the Noetherian polynomial ring
$$
 K[X_n]=K[x_{k,j}\mid 1\le k\le c, 1\le j\le n]
$$
that satisfies $\Sym(n)(I_m)\subseteq I_n$ for all $m\le n$.
Recently, significant advances have been made in this new research direction. For instance, Cohen \cite{Co67, Co87} (see also Aschenbrenner-Hillar \cite{AH07} and Hillar-Sullivant \cite{HS12}) proved that $K[X]$ is $\Sym(\infty)$-Noetherian, i.e., any $\Sym(\infty)$-invariant ideal is generated by finitely many $\Sym(\infty)$-orbits. This result provides a framework as well as a motivation for further studies of  properties of $\Sym(\infty)$-invariant ideals.
In \cite{NR17}, the second and fourth author defined Hilbert series for such ideals and showed the rationality of these series. They obtained in fact a rather explicit formula for the Hilbert series, which enabled them to estimate the asymptotic behavior of  dimensions and degrees of ideals in $\Sym(\infty)$-invariant chains. Their subsequent work \cite{NR17+} extends the aforementioned result of Cohen to the setting of FI-modules (see \cite{CEF} for more details on FI-modules), yielding the stabilization of syzygies in any fixed homological degree of FI-modules. In particular, this applies to any  $\Sym(\infty)$-invariant chain of ideals as discussed above.
For related results the reader may also consult \cite{SaS,Sn}.

In this paper, we study the asymptotic behavior of the Castelnuovo-Mumford regularity of ideals in $\Sym(\infty)$-invariant chains. Recall that the \emph{Castelnuovo-Mumford regularity} (or \emph{regularity} for short) of a finitely generated graded module $M$ over the standard graded polynomial ring $S=K[y_1,\dots,y_m]$ is defined to be
\[
\reg(M)=\max\{j-i\mid \Tor^S_i(K,M)_j\neq 0 \}.
\]
See, e.g., \cite{Ei95,Ei05} for detailed discussions on this important invariant and the role that it plays in algebraic geometry and commutative algebra.

Our first main result (see Theorem \ref{thm_bounding_reg_monomial}) implies that for any $\Sym(\infty)$-invariant chain $(I_n)_{n\ge1}$, $\reg(I_n)$ is eventually bounded above by a linear function, that is, there are numbers $C$ and $D$ such that
\[
\reg(I_n)\le Cn+D\quad \text{for all }\ n\gg 0.
\]
In fact, using a result in \cite{HT} it is not too difficult to show that there is an upper linear bound (see Remark \ref{rem:hoa-trung}). Our contribution is that
we establish a  much better and rather sharp bound for $\reg(\ini_{\le}(I_n))$, where $\le$ is a suitable monomial order on $K[X]$. Note that the chain $(\ini_{\le}(I_n))_{n\ge 1}$ is not $\Sym(\infty)$-invariant in general, but it is invariant under the action of the monoid $\Inc(\N)^i$ of increasing functions on $\N$ for any integer $i\ge 0$ (see Section \ref{sec2} for the precise definition). We will therefore work in a more general setting of $\Inc(\N)^i$-invariant chains.

When the ring $K[X]$ has only one row of variables (that is, $c=1$), we are led to several interesting consequences. For instance,  if $I$ is an $\Inc(\N)$-invariant monomial ideal in $K[X]$ that contains at least one squarefree monomial, e.g., a squarefree monomial ideal, then $\reg(I_n)$ is eventually constant, where $I_n=I\cap K[X_n]$ for all $n\ge 1$ (see Proposition \ref{c=1_constant}).

The mentioned results suggest the following expectation:
\begin{conj}
 \label{conj}
 Let $i\ge 0$ be an integer and $(I_n)_{n\ge 1}$ an $\Inc(\N)^i$-invariant chain of graded ideals. Then $\reg(I_n)$ is eventually a linear function, that is,
\[
\reg(I_n)=Cn+D \quad \text{for some integer constants } C,\; D\ \text{ whenever }\ n\gg 0.
\]
\end{conj}

This conjecture is reminiscent of the well-known asymptotic behavior of the regularity of powers of a graded ideal in a Noetherian polynomial ring \cite{CHT,K}. Note, however, that the methods used in \cite{CHT,K}, which are based on the Noetherianity of the Rees algebra, do not apply to our situation. 

The second main result of the paper establishes Conjecture~\ref{conj} for families of chains that are extremal in a certain sense (see Theorem \ref{thm_extremal}). As an application, we show that this conjecture is true for invariant chains whose members are eventually Artinian (see Corollary \ref{thm: artinian chains}).

It is worth mentioning that our study of the regularity of $\Inc(\N)^i$-invariant chains has an immediate, yet rather surprising implication for the existence of generic initial chains of such chains. Generic initial chains of $\Inc(\N)^i$-invariant chains, if existed, would be a very useful tool for studying equivariant Hilbert series and other invariants of interest. However, as we will show, they do not exist in general (see Proposition \ref{generic}). 

Many of the techniques developed in this paper are not specific to the Castelnuovo-Mumford regularity. Potentially, they could be used to analyze other invariants. Indeed, for the codimension and projective dimension this has been carried out in \cite{LNNR}. 

The paper is organized as follows. In the next section we set up notation and review some basic facts concerning $\Inc(\N)^i$-invariant chains. \Cref{sec3} is devoted to defining some invariants that are used to provide a sharp upper bound on the growth of regularity along such chains. This bound and its consequences are given in \Cref{sec4}, while technical proofs are postponed until \Cref{sec5}. In \Cref{sec6} we verify Conjecture~\ref{conj} for extremal chains, and particularly, for chains whose ideals are eventually Artinian. Finally, the existence of generic initial chains of $\Inc(\N)^i$-invariant chains is discussed in \Cref{sec7}.

{\bf Acknowledgements.}
 We would like to thank the anonymous referees for insightful comments and suggestions that helped us to significantly improve the quality and clarity of the paper. The second author was partially supported by Simons Foundation grant $\#$317096. The third author was partially supported by Project CT 0000.03/19-21 of Vietnam Academy of Science and Technology (VAST); he acknowledges the support of Project ICRTM01$\_$2019.01 of the International Centre for Research and Postgraduate Training in Mathematics (ICRTM), Institute of Mathematics, VAST.


\section{Invariant chains of ideals}\label{sec2}

In this section we recall relevant notions and basic facts concerning invariant chains of ideals. 
Let $\N$ denote the set of positive integers. For any $n\in \N$, set $[n]=\{1,\dots,n\}$. In the sequel we fix a field $K$ and a positive integer $c\in \N$. Consider the following sets of indeterminates
\[
 \begin{aligned}
  X_n&=\{x_{k,j}\mid  k\in [c], j\in [n]\}\quad \text{for }n \ge 1,\\
  X&=\bigcup_{n\ge1}X_n=\{x_{k,j}\mid k\in [c], j\in \N\}.
 \end{aligned}
\]
Denote by $R_n=K[X_n]$ and $R=K[X]$ the polynomial rings over the field $K$ with indeterminates $X_n$ and $X$, respectively. Thus, we have a chain
\[
R_1\subseteq R_2\subseteq\cdots\subseteq R_n\subseteq\cdots
\]
of Noetherian subrings of the non-Noetherian ring $R$.

For $n \ge 1$, let $\Sym(n)$ be the symmetric group on $[n]$. Then $\Sym(n)$ can be regarded as the subgroup of $\Sym(n+1)$ that fixes $n+1$. Let $\Sym(\infty)=\bigcup_{n\ge1}\Sym(n)$ be the group of all finite permutations of $\N$, i.e., permutations that fix all but finitely many elements of $\N$. Recall that $\Sym(\infty)$ acts on $R$ by permuting the second index of the variables:
\[
 \sigma \cdot x_{k,j}=x_{k,\sigma(j)} \quad\text{for }\ \sigma\in\Sym(\infty),\ 1\le k\le c,\ j\ge 1.
\]
This action induces an action of $\Sym(n)$ on $R_n$ for every $n \ge 1$.

An ideal $I\subseteq R$ is said to be \emph{$\Sym(\infty)$-invariant} if $\sigma(f)\in I$ for all $f\in I$ and $\sigma\in\Sym(\infty)$. Any {$\Sym(\infty)$-invariant} ideal gives rise to a $\Sym(\infty)$-invariant chain, and vice versa. By a \emph{$\Sym(\infty)$-invariant chain} we mean a sequence $(I_n)_{n\ge 1}$ of ideals $I_n\subseteq R_n$ satisfying
\[
 \Sym(n)(I_m)=\{\sigma(f)\mid f\in I_m,\ \sigma\in\Sym(n)\}\subseteq I_n\quad \text{whenever }\ m\le n.
\]
It is clear that for any {$\Sym(\infty)$-invariant chain} $(I_n)_{n\ge 1}$, the union $I=\bigcup_{n\ge 1}I_n$ is a {$\Sym(\infty)$-invariant} ideal in $R$. Conversely, if $I$ is a {$\Sym(\infty)$-invariant} ideal, then the sequence of its truncations $I_n=I\cap R_n$ forms a $\Sym(\infty)$-invariant chain, called the \emph{saturated chain} of $I$. Note that an arbitrary {$\Sym(\infty)$-invariant chain} $(I_n)_{n\ge 1}$ is a subchain of the saturated chain of $I=\bigcup_{n\ge 1}I_n$. In other words, the saturated chain is the largest chain among all $\Sym(\infty)$-invariant chains that give rise to the same {$\Sym(\infty)$-invariant} ideal. The reader may consult \cite{NR17+} for an alternative point of view of $\Sym(\infty)$-invariant chains/ideals in the context of FI-modules.

A difficulty when working with {$\Sym(\infty)$-invariant} ideals is that $\Sym(\infty)$ behaves badly with respect to monomial orders on $R$ (see \cite[Remark 2.1]{BD11}). In particular, the initial ideal of a $\Sym(\infty)$-invariant ideal is not necessary a $\Sym(\infty)$-invariant ideal (see \cite[Example 2.2]{LNNR}). To overcome this difficulty, one introduces the following monoid of increasing functions on $\N$:
\[
 \Inc(\N) = \{\pi: \N \to \N \mid \pi(j)<\pi(j+1)\ \text{ for all }\ j\ge 1\}.
\]
In this paper, submonoids of $\Inc(\N)$ that fix initial segments of $\N$ also play an important role. Setting
\[
\Inc(\N)^i=\{\pi \in \Inc(\N) \mid \pi(j)=j\ \text{ for all }\ j\le i\}
\]
for any integer $i\ge0$, we obtain a descending chain of monoids
\[
\Inc(\N)=\Inc(\N)^0 \supset \Inc(\N)^1 \supset \Inc(\N)^2 \supset \cdots.
\]

As for $\Sym(\infty)$, one defines analogously the action of $\Inc(\N)^i$ on $R$ as well as $\Inc(\N)^i$-invariant chains/ideals. A chain $(I_n)_{n\ge 1}$ with $I_n$ an ideal in $R_n$ is called \emph{$\Inc(\N)^i$-invariant} if
\[
 \Inc(\N)^i_{m,n}(I_m)\subseteq I_n\quad \text{whenever }\ m\le n,
\]
where
\[
\Inc(\N)^i_{m,n}=\{\pi \in \Inc(\N)^i \mid \pi(m)\le n\}.
\]
Even though $\Inc(\N)^i$ is not a submonoid of $\Sym(\infty)$, it is easily seen that for any $f\in R_m$ and any $\pi\in\Inc(\N)^i_{m,n}$ with $m\le n$, there exists $\sigma\in\Sym(n)$ such that $\pi f=\sigma f$ (see, e.g., \cite[Lemma 7.6]{NR17}). Hence, $\Inc(\N)^i_{m,n}\cdot f\subseteq\Sym(n) \cdot f$.
It follows that every $\Sym(\infty)$-invariant chain/ideal in $R$ is also $\Inc(\N)^i$-invariant.

 Although the ring $R$ is not Noetherian, it has a very useful finiteness property. Cohen \cite[Theorem 7]{Co87} (see also Hillar-Sullivant \cite[Theorem 3.1]{HS12} and Nagel-R\"{o}mer \cite[Corollary 3.6]{NR17}) showed that $R$ is \emph{$\Inc(\N)^i$-Noetherian} (and hence, $\Sym(\infty)$-Noetherian), in the sense that any $\Inc(\N)^i$-invariant ideal $I\subseteq R$ is generated by the $\Inc(\N)^i$-orbits of finitely many elements, i.e., there exist $f_1,\dots,f_k\in I$ such that $I=\langle\Inc(\N)^i\cdot f_1,\dots,\Inc(\N)^i\cdot f_k\rangle$. From this result it follows that each $\Inc(\N)^i$-invariant chain $\Icc=(I_n)_{n\ge 1}$ \emph{stabilizes}, meaning that there is an integer $r\ge1$ such that for all $n\ge m\ge r$,
\[
I_n=\langle\Inc(\N)^i_{m,n}(I_m)\rangle_{R_n}
\]
as ideals in $R_n$ (see \cite[Lemma 5.2, Corollary 5.4]{NR17}). The least integer $r$ with this property is called the \emph{$i$-stability index} of $\Icc$, denoted by $\ind^i(\Icc)$. Obviously, an $\Inc(\N)^i$-invariant chain $\Icc$ is also $\Inc(\N)^{i+1}$-invariant. One can show that (see \cite[Corollary 6.6]{NR17})
\[
 \ind^{i+1}(\Icc)\le \ind^i(\Icc)+1.
\]

A key advantage of the monoids $\Inc(\N)^i$ over the group $\Sym(\infty)$, especially when working with invariant chains/ideals, is that the monoids $\Inc(\N)^i$ behave well with respect to certain monomial orders on $R$. We say that a monomial order $\le$ \emph{respects} $\Inc(\N)^i$ if $\pi(u)\le \pi(v)$ whenever $\pi \in \Inc(\N)^i$ and $u,v$ are monomials of $R$ with $u\le v$. This condition implies that
\[
\ini_{\le}(\pi(f))=\pi(\ini_{\le}(f))\quad \text{for all } f\in R\ \text{ and } \pi \in \Inc(\N)^i.
\]
An example of a monomial order respecting $\Inc(\N)^i$ is the lexicographic order $\le$ on $R$ induced by the following ordering of the variables:
\[
x_{k,j}\le x_{k',j'} \quad \text{if either }\ k<k'\ \text{ or }\ k=k'\ \text{ and }\ j<j'.
\]
In this paper, whenever $\le$ is a monomial order on $R$, we will use the same notation to denote its restriction to a subring $R_n$.

\begin{lem}[{\cite[Lemma 7.1]{NR17}}]
	\label{lem_initial_filtration}
	Let $\Icc=(I_n)_{n\ge 1}$ be an $\Inc(\N)^i$-invariant chain of ideals. Then for any monomial order $\le$ respecting $\Inc(\N)^i$ the chain $\ini_{\le}(\Icc)=(\ini_{\le}(I_n))_{n\ge 1}$ is also $\Inc(\N)^i$-invariant and
	\[
	\ind^i(\Icc) \le \ind^i(\ini_{\le}(\Icc)).
	\]
\end{lem}

This simple lemma allows us to reduce the problem of bounding the Castelnuovo-Mumford regularity of ideals in an invariant chain to the case of monomial ideals. For the statement of our result in this case we need some further notation, which will be introduced in the next section.

\section{Weight functions}\label{sec3}

This section is devoted to defining certain invariants that we use to bound the growth of the Castelnuovo-Mumford regularity. Throughout the section, let $i \ge 0$ be a fixed integer and $\Icc=(I_n)_{n\ge 1}$ a nonzero $\Inc(\N)^i$-invariant chain of monomial ideals with $r=\ind^i(\Icc)$.
 
 For any monomial $1 \neq u \in R_n=K[X_n]$, we denote by $\min(u)$ (respectively, $\max(u)$) the smallest (respectively, largest) index $j$ such that $x_{k,j}$ divides $u$ for some $k\in [c]$. When $u=1$, we adopt the convention that $\min(u)=\max(u)=0$. Let $G(I_n)$ be the minimal set of monomial generators of $I_n$. In order to analyze the action of $\Inc(\N)^i$, we partition $G(I_n)$ into the following subsets:
\begin{align*}
 G^{(i)}_+(I_n)&=\{u\in G(I_n)\mid \min(u)>i\},\\
 G^{(i)}(I_n)&=\{u\in G(I_n)\mid \min(u)\le i <\max(u)\},\\
 G^{(i)}_-(I_n)&=\{u\in G(I_n)\mid \max(u)\le i\}.
\end{align*}

For all $n\ge r$ it is evident that $G^{(i)}_-(I_n)=G^{(i)}_-(I_r)$ since the elements of $G^{(i)}_-(I_r)$ are fixed by $\Inc(\N)^i$. So intuitively, the growth rate of $\reg(I_n)$ should depend only on the set $G^{(i)}_+(I_r)\cup G^{(i)}(I_r)$. Moreover, this growth rate might behave differently depending on whether $G^{(i)}_+(I_r)\ne \emptyset$ or $G^{(i)}_+(I_r)= \emptyset$, as illustrated by the following example.

\begin{example}
	\label{ex31}
	Consider the case $c=1$, i.e., $R$ has only one row of variables. For simplicity we take $K=\mathbb{Q}$ and write $R=\mathbb{Q}[x_j\mid j\ge 1].$ Let
	\[
	I_3=\langle x_1^2x_2,\ x_1x_3^5,\ x_2x_3^4,\ x_2^3x_3\rangle\subset R_3.
	\]
	Then $G^{(1)}_+(I_3)=\{x_2x_3^4,\ x_2^3x_3\}$. Consider the $\Inc(\N)^1$-invariant chain $\Icc=(I_n)_{n\ge 1}$ with
	\[
	I_n=\langle\Inc(\N)^1_{3,n}(I_3)\rangle \quad \text{for all }\ n\ge 3.
	\]
	($I_1$ and $I_2$ are unimportant, one may set $I_1=I_2=0$.) Using induction it is easy to see that
	\[
	I_n=\langle x_1^2x_{j-1},\ x_1x_j^5\mid 3\le j\le n\rangle +\langle x_jx_k^4,\ x_j^3x_k\mid 2\le j<k\le n\rangle \quad \text{for all }\ n\ge 3.
	\]
	Computations with Macaulay2 \cite{GS} yield the following table, which suggests that $\reg(I_n)$ could be a linear function with leading coefficient $2$ when $n\ge3$:
	\begin{table}[h]
		\begin{tabular}{|c|c|c|c|c|c|c|c|c|}
			\hline
			\multicolumn{1}{|c|}{$n$}& 3& 4 & 5 & 6 & 7  & 8  & 9  & 10 \\ \hline
			$\reg(I_n)$             &7& 9& 11 & 13 & 15 & 17 & 19& 21 \\ \hline
		\end{tabular}
	\end{table}
	
	We predict that the coefficient 2 can be determined by $G^{(1)}_+(I_3)$ as follows. For a nonzero monomial $u\in R$ let $w^{(1)}(u)$ be the largest exponent $e$ such that $x_{j}^e$ divides $u$ for some $j >1$ (a more general definition is given below). So for instance, $w^{(1)}(x_2x_3^4)=4$ and $w^{(1)}(x_2^3x_3)=3$. Now we observe the following relation between the leading coefficient 2 and the set $G^{(1)}_+(I_3)$:
	\[
	 2=\min\{w^{(1)}(u)-1\mid u\in G^{(1)}_+(I_3)\}.
	\]
	We also observe the same relation for many other $\Inc(\N)^i$-invariant chains with $G^{(i)}_+(I_r)\ne \emptyset$, which leads us to Conjecture \ref{con_c=1} in the next section.
	
	When $G^{(i)}_+(I_r)= \emptyset$ and $G^{(i)}(I_r)\ne \emptyset$, however, the leading coefficient of the predicted linear function $\reg(I_n)$ might behave differently. Computational experiments suggest that this coefficient should belong to the set $\{w^{(i)}(u)-1\mid u\in G^{(i)}(I_r)\}$, but is not necessarily the minimal one. For example, consider the ideal
	\[
	I_3'=\langle x_1^2x_2,\ x_1x_3^5\rangle\subset R_3
	\]
	and the $\Inc(\N)^1$-invariant chain $\Icc'=(I_n')_{n\ge 1}$ with
	\[
	I_n'=\langle\Inc(\N)^1_{3,n}(I_3')\rangle \quad \text{for all }\ n\ge 3.
	\]
	The following table is computed by Macaulay2:
	\begin{table}[h]
		\begin{tabular}{|c|c|c|c|c|c|c|c|c|}
			\hline
			\multicolumn{1}{|c|}{$n$}& 3& 4 & 5 & 6 & 7  & 8  & 9  & 10 \\ \hline
			$\reg(I_n')$             &7& 10& 14 & 18 & 22 & 26 & 30& 34 \\ \hline
		\end{tabular}
	\end{table}
	
	\noindent Thus, for $n\ge 4$, $\reg(I_n')$ could be a linear function with leading coefficient
	\[
	 4\in\{4,0\}=\{w^{(1)}(u)-1\mid u\in G^{(1)}(I_3')\}.
	\]
\end{example}

Example \ref{ex31} motivates the following definition.

\begin{defn}
	\label{def:i weight}
	Let $i \ge 0$ be a fixed integer and $k\in [c]$. For a nonzero monomial $u\in R_n$ set
	\[\begin{aligned}
 	w_k^{(i)}(u)&=\max\{e\mid \text{$x_{k,j}^e$ divides $u$ for some $j >i$}\},\\
	w^{(i)}(u)&=\max\{w_k^{(i)}(u)\mid k\in [c]\}.
	\end{aligned}	
	\]
	
	If $I_n\subseteq R_n$ is a nonzero monomial ideal, define the following \emph{weights}:
		\[
		\begin{aligned}
		w_k^{(i)}(I_n) &= \max \{ w_k^{(i)}(u) \mid u \in G(I_n)\},\\
		\omega^{(i)}(I_n) &=
		\begin{cases}
		\min \{ w^{(i)}(u) \mid u \in G^{(i)}_+(I_n)\}&\text{if }\ G^{(i)}_+(I_n)\ne\emptyset,\\
		\max \{ w^{(i)}(u) \mid u \in G(I_n)\}&\text{if }\  G^{(i)}_+(I_n)=\emptyset.
		\end{cases}
		\end{aligned}
		\]
\end{defn}

From the above definition one sees immediately that $w^{(i)}(u)=0$ if $u\in G^{(i)}_-(I_n)$. Hence, if $G^{(i)}_+(I_n)=\emptyset$ and $G^{(i)}(I_n)\ne\emptyset$, then
\[
 \omega^{(i)}(I_n) =\max \{ w^{(i)}(u) \mid u \in G(I_n)\}=\max \{ w^{(i)}(u) \mid u \in G^{(i)}(I_n)\}.
\]

\begin{example}
\label{ex33}
Let $i=2$ and $c= 3$. Consider the ideal $I_6=\langle u_1,\; u_2,\; u_3,\; u_4,\; u_5\rangle \subset R_6$ with
\[
u_1= x_{2,1}^4,\; u_2=x_{3,2}x_{1,3}^2x_{2,4},\; u_3=x_{1,1}^3x_{1,4}^2,\; u_4=x_{1,2}x_{2,5}^5x_{3,3}^2,\; u_5=x_{2,3}^3x_{3,4}^2x_{3,5}.
\]
Then 
$
 G^{(2)}_-(I_6)=\{u_1\},$ $G^{(2)}(I_6)=\{u_2,\; u_3,\; u_4\},$ and $G^{(2)}_+(I_6)=\{u_5\}.$
By definition,
\[
\begin{aligned}
  w_k^{(2)}(u_1)&=w^{(2)}(u_1)=0\quad \text{for all } k\in [3],\\
  w_1^{(2)}(u_2)&=2,\ w_2^{(2)}(u_2)=1,\ w_3^{(2)}(u_2)=0, \ w^{(2)}(u_2)=2,\\
  w_1^{(2)}(u_3)&=2,\ w_2^{(2)}(u_3)=0,\ w_3^{(2)}(u_3)=0, \ w^{(2)}(u_3)=2,\\
  w_1^{(2)}(u_4)&=0,\ w_2^{(2)}(u_4)=5,\ w_3^{(2)}(u_4)=2, \ w^{(2)}(u_4)=5,\\
  w_1^{(2)}(u_5)&=0,\ w_2^{(2)}(u_5)=3,\ w_3^{(2)}(u_5)=2, \ w^{(2)}(u_5)=3.
\end{aligned}
\]
It follows that
\[
 w_1^{(2)}(I_6)=2,\ w_2^{(2)}(I_6)=5,\ w_3^{(2)}(I_6)=2, \ \omega^{(2)}(I_6)=3.
\]
On the other hand, if we take $I_6'=\langle u_1,\; u_2,\; u_3,\; u_4\rangle$, then $G^{(2)}(I_6')=\{u_2,\; u_3,\; u_4\}$ and $G^{(2)}_+(I_6')=\emptyset.$ Hence,
\[
 \omega^{(2)}(I_6')=\max\{w^{(2)}(u) \mid u \in G(I_6')\}=\max\{w^{(2)}(u) \mid u \in G^{(2)}(I_6')\}=5.
\]
\end{example}

For later use we record here the following relationship between the weights.

\begin{lem}
 \label{lem_compare}
 Let $I_n\subseteq R_n$ be a nonzero monomial ideal. Then
 \[
  \omega^{(i)}(I_n)\le \max\{w_k^{(i)}(I_n)\mid k\in[c]\},
 \]
with equality if $G^{(i)}_+(I_n)=\emptyset.$
\end{lem}

\begin{proof}
 It follows from Definition \ref{def:i weight} that
 \[ 
 \begin{aligned}
  \omega^{(i)}(I_n)&\le \max \{ w^{(i)}(u) \mid u \in G(I_n)\}=\max \{ \max\{w_k^{(i)}(u)\mid k\in[c]\} \mid u \in G(I_n)\}\\
  &=\max \{ \max\{w_k^{(i)}(u)\mid u \in G(I_n)\} \mid k\in[c]\}=\max\{w_k^{(i)}(I_n)\mid k\in[c]\},
 \end{aligned}
 \]
with equality if $G^{(i)}_+(I_n)=\emptyset.$
\end{proof}

The weights defined in Definition \ref{def:i weight} can be extended to invariant chains. For this purpose, we need the following observation.

\begin{lem}
\label{lem_decreasing}
	Let $\Icc=(I_n)_{n\ge 1}$ be a nonzero $\Inc(\N)^i$-invariant chain of monomial ideals with $r=\ind^i(\Icc)$. Then for every $k\in[c]$ and $n\ge r$ one has
	\[
	\begin{aligned}
	 w_k^{(i)}(I_{n+1})\le w_k^{(i)}(I_{n})\quad\text{and}\quad
	\omega^{(i)}(I_{n+1})\le \omega^{(i)}(I_{n}).
	\end{aligned}\]
\end{lem}

\begin{proof}
 Let $n\ge r$. We first show the inequality $w_k^{(i)}(I_{n+1})\le w_k^{(i)}(I_{n})$ for every $k\in[c]$. Since $I_{n+1}=\langle\Inc(\N)^i_{n,n+1}(I_n)\rangle=\langle\Inc(\N)^i_{n,n+1}(G(I_n))\rangle$, one gets
 \begin{equation}
 \label{eq31}
  G(I_{n+1})\subseteq \Inc(\N)^i_{n,n+1}(G(I_n)).
 \end{equation}
 For any $u\in G(I_n)$ and $\pi\in\Inc(\N)^i_{n,n+1}$ it follows from the definition that 
 \[
  w_k^{(i)}(u)=w_k^{(i)}(\pi(u)).
 \]
 This together with \eqref{eq31} gives
 \[
  \{w_k^{(i)}(v)\mid v\in G(I_{n+1})\}\subseteq \{w_k^{(i)}(u)\mid u\in G(I_{n})\}.
 \]
 Hence,
 \[
  w_k^{(i)}(I_{n+1})=\max\{w_k^{(i)}(v)\mid v\in G(I_{n+1})\}\le \max\{w_k^{(i)}(u)\mid u\in G(I_{n})\}=w_k^{(i)}(I_{n}).
 \]
 
 For the inequality $\omega^{(i)}(I_{n+1})\le \omega^{(i)}(I_{n})$ we may assume that $G^{(i)}(I_n)\cup G^{(i)}_+(I_n)\ne \emptyset$, since otherwise $G(I_n)=G^{(i)}_-(I_n)$ is fixed by $\Inc(\N)^i$, yielding $G(I_{n+1})=G(I_n)$, and thus there is nothing to prove. Consider the following cases:
 
 {\em Case 1}: $G^{(i)}_+(I_n)\ne \emptyset$. Then $\omega^{(i)}(I_{n})=\min\{w^{(i)}(u)\mid u\in G^{(i)}_+(I_{n})\}$. Let $u_0\in G^{(i)}_+(I_{n})$ be such that $\omega^{(i)}(I_{n})=w^{(i)}(u_0)$. Since $u_0\in I_n\subseteq I_{n+1}$, $u_0$ is divisible by some $v_0\in G(I_{n+1})$. It is clear that $\min(v_0)\ge \min(u_0)>i$. Thus, $v_0\in G^{(i)}_+(I_{n+1})$, and hence, $G^{(i)}_+(I_{n+1})\ne \emptyset$. So $\omega^{(i)}(I_{n+1})=\min\{w^{(i)}(v)\mid v\in G^{(i)}_+(I_{n+1})\}$ by definition. Therefore,
 \[
  \omega^{(i)}(I_{n+1})\le w^{(i)}(v_0)\le w^{(i)}(u_0)= \omega^{(i)}(I_{n}).
 \]

 {\em Case 2}: $G^{(i)}_+(I_n)= \emptyset$ and $G^{(i)}(I_n)\ne \emptyset$. Then it follows from \eqref{eq31} that $G^{(i)}_+(I_{n+1})= \emptyset$ and
 \begin{align*}
  G^{(i)}(I_{n+1})&\subseteq \Inc(\N)^i_{n,n+1}(G^{(i)}(I_n)).
 \end{align*}
 Arguing as in the proof of the inequality $w_k^{(i)}(I_{n+1})\le w_k^{(i)}(I_{n})$, with $w_k^{(i)}(u)$ replaced by $w^{(i)}(u)$ and $G(I_n)$ replaced by $G^{(i)}(I_n)$, we get
 \begin{equation}
 \label{eq32}
  \{w^{(i)}(v)\mid v\in G^{(i)}(I_{n+1})\}\subseteq \{w^{(i)}(u)\mid u\in G^{(i)}(I_{n})\},
 \end{equation}
which yields
\[
  \omega^{(i)}(I_{n+1})=\max\{w^{(i)}(v)\mid v\in G^{(i)}(I_{n+1})\}\le\max\{w^{(i)}(u)\mid u\in G^{(i)}(I_{n})\}=\omega^{(i)}(I_{n}),
 \]
 as desired.
\end{proof}

\begin{remark}
 Under the assumption of the previous lemma, it is not hard to show that if $G^{(i)}_+(I_r)\ne \emptyset$, then $\omega^{(i)}(I_{n})= \omega^{(i)}(I_{r})$ for all $n\ge r$. We leave the details to the interested reader.
\end{remark}

Lemma \ref{lem_decreasing} says that $(w_k^{(i)}(I_{n}))_{n\ge r}$ and $(\omega^{(i)}(I_{n}))_{n\ge r}$ are non-increasing sequences of non-negative integers. So they must be eventually constants. This justifies the following definition.

\begin{defn}
	\label{def:i weight2}
	Let $\Icc=(I_n)_{n\ge 1}$ be a nonzero $\Inc(\N)^i$-invariant chain of monomial ideals. Set
	\[
	\begin{aligned}
	w_k^{(i)}(\Icc)&=w_k^{(i)}(I_{n})\quad \text{for } n\gg 0,\\
	\omega^{(i)}(\Icc)&= \omega^{(i)}(I_{n})\quad \text{for } n\gg 0.
	\end{aligned}\]
\end{defn}

\begin{example}
\label{ex_i_weights}
 Let $i=1$ and $c=2$. Consider the ideal
 \[
  I_4=\langle x_{1,2}^4x_{2,2},\; x_{1,3}^2,\; x_{2,2}^3,\; x_{1,1}x_{2,3}^5,\; x_{2,4}^4 \rangle.
 \]
 Then
 \[
  w_1^{(1)}(I_{4})=4,\quad w_2^{(1)}(I_{4})=5,\quad \omega^{(1)}(I_{4})=2.
 \]
 Let $\Icc=(I_n)_{n\ge 1}$ be the $\Inc(\N)^1$-invariant chain with $I_n=\langle\Inc(\N)^1_{4,n}(I_4)\rangle$ for all $n\ge 4$. One can show by induction that
  \begin{equation}
  \label{eq38}
   I_n=\langle x_{1,2}^4x_{2,2},\; x_{1,j}^2\mid 3\le j\le n-1\rangle+\langle x_{2,j}^3,\;x_{2,n-1}^4,\;x_{2,n}^4\mid 2\le j\le n-2\rangle \ \ \text{for all } n\ge 5.
  \end{equation}
 Thus, for $n\ge 5$ it holds that
 \[
  w_1^{(1)}(\Icc)=w_1^{(1)}(I_{n})=4,\quad  w_2^{(1)}(\Icc)=w_2^{(1)}(I_{n})=4,\quad \omega^{(1)}(\Icc)=\omega^{(1)}(I_{n})=2.
 \]

\end{example}

The weights introduced in Definition \ref{def:i weight2} have further stabilization properties. Recall that an $\Inc(\N)^i$-invariant chain $\Icc=(I_n)_{n\ge 1}$ is also $\Inc(\N)^{i+1}$-invariant. By definition, it is clear that $w_k^{(i+1)}(u)\le w_k^{(i)}(u)$ and $w^{(i+1)}(u)\le w^{(i)}(u)$ for every nonzero monomial $u$ and every $k\in [c]$. This implies
\[
 w_k^{(i+1)}(\Icc) \le w_k^{(i)}(\Icc)\quad \text{and}\quad\omega^{(i+1)}(\Icc)\le \omega^{(i)}(\Icc).
\]
Thus, $w_k^{(i)}(\Icc)$ and $\omega^{(i)}(\Icc)$ must be constants for $i\gg0$, leading to the following:
\begin{defn}
	\label{def:i weight3}
	Let $\Icc=(I_n)_{n\ge 1}$ be a nonzero $\Inc(\N)^i$-invariant chain of monomial ideals for some $i\ge0$. We define
	\[
	\begin{aligned}
	w_k(\Icc)&=w_k^{(i)}(\Icc)\quad \text{for } i\gg 0,\\
	\omega(\Icc)&= \omega^{(i)}(\Icc)\quad \text{for } i\gg 0.
	\end{aligned}\]
\end{defn}

\begin{example}
 Consider again the $\Inc(\N)^1$-invariant chain $\Icc=(I_n)_{n\ge 1}$ in Example \ref{ex_i_weights}. Then from \eqref{eq38} it is easily seen that
 \[
  w_1^{(i)}(\Icc)=2,\quad  w_2^{(i)}(\Icc)=4,\quad \omega^{(i)}(\Icc)=2\quad \text{for all } i\ge 2.
 \]
Hence, $w_1(\Icc)=2$, $w_2(\Icc)=4$, and $\omega(\Icc)=2$.
\end{example}


\section{Upper bounds for Castelnuovo-Mumford regularity}\label{sec4}

We continue to use the notation introduced in the previous sections. In particular, we fix a positive integer $c$ and consider the polynomial ring $R=K[x_{k,j}\mid k\in [c], j\in \N]$ and its subrings $R_n=K[x_{k,j}\mid k\in [c], j\in [n]]$ for $n\ge 1$. 
With the notation of Definition \ref{def:i weight3} the main result of this section is:

\begin{thm}
	\label{thm_bounding_reg_monomial}
	Let $\Icc=(I_n)_{n\ge 1}$ be a nonzero $\Inc(\N)^i$-invariant chain of monomial ideals. Set 
	\[
	C(\Icc)=\max\{\omega(\Icc)-1,0\}+\max\Big\{\sum_{k\ne l}w_k(\Icc)\mid l\in[c]\Big\}. 
	\]
	Then
	\[
	\reg(I_n)\le C(\Icc)n+D(\Icc)\quad \text{for all }\ n\gg 0,
	\]
	where $D(\Icc)$ is a suitable constant.
\end{thm}

We postpone the proof of Theorem \ref{thm_bounding_reg_monomial} until the next section. For now, let us derive from it some important consequences. At first, we obtain a weaker version of the theorem that is easier to use.

\begin{cor}
	\label{cor_i_bound}
	Let $\Icc=(I_n)_{n\ge 1}$ be a nonzero $\Inc(\N)^i$-invariant chain of monomial ideals. Set 
	\[
	C^{(i)}(\Icc)=\max\{\omega^{(i)}(\Icc)-1,0\}+\max\Big\{\sum_{k\ne l}w_k^{(i)}(\Icc)\mid l\in[c]\Big\}. 
	\]
	Then there exists a constant $D^{(i)}(\Icc)$ such that
	\[
	\reg(I_n)\le C^{(i)}(\Icc)n+D^{(i)}(\Icc)\quad \text{for all }\ n\gg 0.
	\]
\end{cor}

\begin{proof}
 By definition, $\omega(\Icc)\le \omega^{(i)}(\Icc)$ and $w_k(\Icc)\le w_k^{(i)}(\Icc)$ for all $k\in[c]$. Hence, $C(\Icc)\le C^{(i)}(\Icc)$, and the corollary follows at once from Theorem \ref{thm_bounding_reg_monomial}.
\end{proof}

\begin{remark}
 \label{rem: equi}
 Although Corollary \ref{cor_i_bound} is weaker than Theorem \ref{thm_bounding_reg_monomial}, the correctness of these two results are equivalent. Indeed, assume that Corollary \ref{cor_i_bound} is true. Let $i_0\gg0$ be such that $C(\Icc)= C^{(i_0)}(\Icc)$. Since $\Icc$ is also an $\Inc(\N)^{i_0}$-invariant chain, it follows from Corollary \ref{cor_i_bound} that for all $n\gg 0$,
 \[
 \reg(I_n)\le C^{(i_0)}(\Icc)n+D^{(i_0)}(\Icc)=C(\Icc)n+D^{(i_0)}(\Icc),
 \]
 yielding Theorem \ref{thm_bounding_reg_monomial}. Therefore, in order to prove Theorem \ref{thm_bounding_reg_monomial}, it is enough to prove Corollary \ref{cor_i_bound}.
\end{remark}

\begin{remark}
   \label{rem:hoa-trung}
 Let $\Icc=(I_n)_{n\ge 1}$ be a nonzero $\Inc(\N)^i$-invariant chain of monomial ideals with $r=\ind^i(\Icc)$. Using \cite[Theorem 1.2(i)]{HT} one can show that
 \[
 \reg(I_n)\le \delta(I_r)cn\quad \text{for }\ n\gg 0,
 \]
 where $\delta(I_r)$ denotes the maximal degree of a monomial in $G(I_r)$. Since clearly
 \[
  \delta(I_r)\ge \omega^{(i)}(I_r)\ge \omega^{(i)}(\Icc)\ge \omega(\Icc)\ \text{ and } \ \delta(I_r)\ge w_k^{(i)}(I_r)\ge w_k^{(i)}(\Icc)\ge w_k(\Icc) 
 \]
for all $k\in[c]$ (and usually $\delta(I_r)>\max\{w_k^{(i)}(\Icc)\mid k\in [c]\}\ge\omega^{(i)}(\Icc)$), the bounds in Theorem~\ref{thm_bounding_reg_monomial} and Corollary \ref{cor_i_bound} are considerably stronger than the previous one. In fact, the bound in Theorem \ref{thm_bounding_reg_monomial} is rather sharp, as illustrated by the next example and the discussions in the case $c=1$ below (see, especially, Conjecture \ref{con_c=1} and Propositions \ref{c=1_constant}, \ref{mon_orbit}). 
\end{remark}

\begin{example}
 Let $\Icc=(I_n)_{n\ge 1}$ be an $\Inc(\N)^1$-invariant chain where $c = 2$, $K=\mathbb{Q}$, $I_n = \langle \Inc(\N)^1_{3, n} (I_3) \rangle$ if $n \ge 3$, and 
 \[
  I_3=\langle x_{1,1}x_{1,3},\; x_{1,2}^2x_{2,2}^2\rangle.
 \]
Then one can easily check that
\[
 w_1(\Icc)=w_1^{(1)}(\Icc)=2, \quad w_2(\Icc)=w_2^{(1)}(\Icc)=2, \quad \omega(\Icc)=\omega^{(1)}(\Icc)=2.
\]
Thus,
\[
 C(\Icc)=C^{(1)}(\Icc)=2-1+2=3.
\]
This number is possibly the leading coefficient of a linear function describing $\reg(I_n)$ when $n\gg0$, as suggested by the following table, which was obtained using Macaulay2:
\begin{table}[h]
		\begin{tabular}{|c|c|c|c|c|c|c|c|c|}
			\hline
			\multicolumn{1}{|c|}{$n$}& 3& 4 & 5 & 6 & 7  & 8  & 9  & 10 \\ \hline
			$\reg(I_n)$             &5& 7& 10 & 13 & 16 & 19 & 22& 25 \\ \hline
		\end{tabular}
	\end{table}

More generally, for $c\ge1$, computations with Macaulay2 also suggest that the bound in Theorem~\ref{thm_bounding_reg_monomial} is tight if one considers the chain $\Icc'=(I_n')_{n\ge 1}$ defined as above with $I_3$ replaced by
\[
  I_3'=\langle x_{1,1}x_{1,3},\; x_{1,2}^mx_{2,2}^m\cdots x_{c,2}^m\rangle,
 \]
 where $m$ is a positive integer.
\end{example}

 As another application of Theorem \ref{thm_bounding_reg_monomial} we give an upper linear bound for the asymptotic behavior of the Castelnuovo-Mumford regularity along any $\Inc(\N)^i$-invariant chain of graded ideals.

\begin{cor}
	\label{thm_bounding_reg}
 Let $\Icc=(I_n)_{n\ge 1}$ be a nonzero $\Inc(\N)^i$-invariant chain of graded ideals. Let $\le$ be a monomial order on $R$ respecting $\Inc(\N)^i$. 
 Then there exists a constant $D(\Icc)$ such that
	\[
	\reg(I_n)\le C(\ini_{\le}(\Icc))n+D(\Icc)\quad \text{for all }\ n\gg 0.
	\]
\end{cor}

\begin{proof}
It is well-known (see, e.g., \cite[Theorem 3.3.4]{HH11}) that 
$\reg (I_n) \le \reg (\ini_{\le}(I_n))$ for all $n\ge 1$.
Hence the result follows from Lemma \ref{lem_initial_filtration} and Theorem \ref{thm_bounding_reg_monomial}.
\end{proof}

Recall that any $\Sym(\infty)$-invariant chain is also $\Inc(\N)^i$-invariant. Therefore, as a direct consequence of Corollary \ref{thm_bounding_reg} we get:

\begin{cor}
	\label{cor_boundingreg_Sym}
	Let $\Icc=(I_n)_{n\ge 1}$ be a $\Sym(\infty)$-invariant chain of graded ideals. Then there exist integers $C$ and $D$ such that 
$$\reg (I_n) \le Cn+D \quad\text{ for all }\ n\gg 0.$$
\end{cor}

Next, we discuss some interesting corollaries of Theorem \ref{thm_bounding_reg_monomial} in the case $c=1$, that is, when $R$ has only one row of variables. In this case, we will write $R_n=K[x_1,\dots,x_n]$ and $R=K[x_j\mid j\in\N]$. The next result is an immediate consequence of Theorem \ref{thm_bounding_reg_monomial}.

\begin{cor}
	\label{c=1}
	Suppose $c=1$. Let $\Icc=(I_n)_{n\ge 1}$ be a nonzero $\Inc(\N)^i$-invariant chain of monomial ideals. Then there is a constant $D(\Icc)$ such that
	\[
	\reg(I_n)\le \max\{\omega(\Icc)-1,0\}n+D(\Icc)\quad \text{for all }\ n\gg 0.
	\]
\end{cor}

Keep the assumption of the preceding corollary and let $r= \ind^i(\Icc)$. Since $\omega(\Icc)\le \omega^{(i)}(\Icc)\le\omega^{(i)}(I_r)$ one gets
\begin{equation}
\label{eq_c=1}
 \reg(I_n)\le \max\{\omega^{(i)}(I_r)-1,0\}n+D(\Icc)\quad \text{for all }\ n\gg 0.
\end{equation}
By definition, it is apparent that $\omega^{(i)}(I_r)=0$ if and only if
\begin{enumerate}
 \item[(a)]
 $G(I_r)=G^{(i)}_-(I_r)$. 
\end{enumerate}
Moreover, $\omega^{(i)}(I_r)=1$ if and only if one of the following conditions is satisfied:
\begin{enumerate}
 \item[(b)]
 $G^{(i)}_+(I_r)$ contains at least one squarefree monomial,
 \item[(c)]
 $G^{(i)}_+(I_r)=\emptyset$, $G^{(i)}(I_r)\ne \emptyset$ and the elements of $G^{(i)}(I_r)$ are squarefree in the variables $x_{k,j}$ with $k\in [c]$ and $j>i$.
\end{enumerate}
Note also that $G^{(i)}_+(I_r)=G(I_r)$ when $i=0$ and $I_r\ne R_r$. So from \eqref{eq_c=1} we obtain:

\begin{cor}
	\label{c=1_bounded}
	Suppose $c=1$. Let $\Icc=(I_n)_{n\ge 1}$ be a nonzero $\Inc(\N)^i$-invariant chain of monomial ideals with $r=\ind^i(\Icc)$. Assume that $I_r$ satisfies one of Conditions (a)-(c) above. Then the sequence $(\reg(I_n))_{n\ge 1}$ is bounded.
	
	In particular, the sequence $(\reg(I_n))_{n\ge 1}$ is bounded if $\Icc=(I_n)_{n\ge 1}$ is a nonzero $\Inc(\N)$-invariant chain of monomial ideals with $r=\ind^0(\Icc)$ and $I_r$ has at least one squarefree minimal generator.
\end{cor}

\begin{cor}
	Suppose $c=1$. Let $\Icc=(I_n)_{n\ge 1}$ be a nonzero $\Inc(\N)$-invariant chain of graded ideals. Assume that
	$I_r$ has a minimal generator whose terms are squarefree for some $r\ge 1$. Then the sequence $(\reg(I_n))_{n\ge 1}$ is bounded.
	
	In particular, the conclusion is true if $\Icc$ is a $\Sym(\infty)$-invariant chain of graded ideals.
\end{cor}

\begin{proof}
It suffices to prove the first assertion. We may assume that the ideals $I_n$ are proper, since $\reg(R_n)=0$. Let $\le$ be a monomial order on $R$ respecting $\Inc(\N)$. By Lemma \ref{lem_initial_filtration}, the chain $\ini_{\le}(\Icc)=(\ini_{\le}(I_n))_{n\ge 1}$ is also $\Inc(\N)$-invariant. Set $r'=\ind^0(\ini_{\le}(\Icc))$. The assumption implies that $\ini_{\le}(I_r)$ has a squarefree minimal generator. Clearly, no matter whether $r\ge r'$ or $r<r'$ this forces that $\ini_{\le}(I_{r'})$ has a squarefree minimal generator. The assertion now follows from Corollary \ref{c=1_bounded} and the fact that $\reg (I_n) \le \reg (\ini_{\le}(I_n))$.
\end{proof}

In the previous result, the assumption on a minimal generator can sometimes be relaxed. 

\begin{example}
Let $K=\mathbb{Q}$. Consider the $\Inc(\N)$-invariant chain $\Icc=(I_n)_{n\ge 1}$ generated by 
$I_3 = \langle x_1 x_2 + x_3^2 \rangle \subseteq R_3$. Using the lexicographic order $\le$, the 
ideals $\ini_{\le}(I_n)$ contain the monomial $x_1 x_2$ if $n \ge 3$. Thus, the sequences 
$(\reg(\ini_{\le} (I_n)))_{n\ge 1}$ and $(\reg(I_n))_{n\ge 1}$ are bounded by Corollary \ref{c=1_bounded}. A Macaulay2 
computation gives $\reg (I_n) = 4$ if $4 \le n \le 10$, suggesting that $\reg (I_n) = 4$ whenever 
$n \ge 4$. 
\end{example}

As mentioned in Example \ref{ex31}, experiments with Macaulay2 suggest that the bound in Corollary \ref{c=1} is tight if $G_i^+(I_r)\ne \emptyset$. Therefore, we propose the following more precise form of Conjecture \ref{conj} for this case.

\begin{conj}
\label{con_c=1}
Suppose $c=1$. Let $\Icc=(I_n)_{n\ge 1}$ be a nonzero $\Inc(\N)^i$-invariant chain of monomial ideals with $r= \ind^i(\Icc)$. Assume that $G_i^+(I_r)\neq \emptyset$. Then there is a constant $D(\Icc)$ such that
\[
\reg(I_n)=(\omega(\Icc)-1)n+D(\Icc)\quad \text{for }\ n\gg 0.
\]

In particular, the above statement holds if $\Icc=(I_n)_{n\ge 1}$ is a nonzero $\Inc(\N)$-invariant chain of proper monomial ideals.
\end{conj}

We will verify this conjecture for two special classes of saturated chains. Recall that an $\Inc(\N)^i$-invariant chain $\Icc=(I_n)_{n\ge 1}$ is {saturated} if there is an $\Inc(\N)^i$-invariant ideal $I$ in $R$ such that $I_n = I \cap R_n$ for $n\ge 1$. When this is the case, Conjecture \ref{con_c=1} is true if $I$ is either a squarefree monomial ideal or generated by one monomial orbit. For the proofs we will need the following well-known fact on regularity (see, e.g., \cite[Theorem 20.2]{Pe11}).

\begin{lem} 
   \label{reg mod regular element} 
If $u$ is a homogeneous non-zero-divisor on a finitely generated graded $R_n$-module $M$ with $\deg(u)>0$, then
\begin{equation*}
 \reg(M/uM)=\reg(M)+\deg(u)-1. 
\end{equation*}
\end{lem}

\begin{prop}
        \label{c=1_constant}
	Suppose $c=1$. Let $\Icc=(I_n)_{n\ge 1}$ be a nonzero saturated $\Inc(\N)^i$-invariant chain of monomial ideals. If $\omega(\Icc)\le 1$ (e.g., one of Conditions (a)-(c) above Corollary \ref{c=1_bounded} is satisfied), then the sequence $(\reg(I_n))_{n\ge 1}$ is eventually constant. 
	
In particular, the conclusion is true if every $I_n$ is a squarefree monomial ideal. 
\end{prop}

\begin{proof} 
 Since $\Icc$ is saturated, it is easy to see that
 \[
  \langle I_{n+1},x_{n+1}\rangle=\langle I_{n},x_{n+1}\rangle \quad \text{for all }\ n\ge \ind^i(\Icc).
 \]
Hence, Lemma \ref{cor:bdl} below gives
\[
 \reg(I_{n+1})\ge \reg(\langle I_{n+1},x_{n+1}\rangle)=\reg(\langle I_{n},x_{n+1}\rangle)=\reg (I_n),
\]
where the last equality follows from Lemma \ref{reg mod regular element} since $x_{n+1}$ is a non-zero-divisor on $R_{n+1}/\langle I_n \rangle$. On the other hand, the sequence $(\reg(I_n))_{n\ge 1}$ is bounded by Corollary \ref{c=1}. Therefore, this sequence must be eventually constant.
\end{proof}

\begin{remark}
 Conca and Varbaro \cite[Corollary 2.7]{CV} have shown that for a graded ideal $J$ in a Noetherian polynomial ring $S$ and a monomial order $\le$ on $S$, one has $\reg(J)=\reg(\ini_{\le}(J))$ if $\ini_{\le}(J)$ is a squarefree monomial ideal. So if $c=1$ and $\Icc=(I_n)_{n\ge 1}$ is an $\Inc(\N)^i$-invariant chain of graded ideals such that $(\ini_{\le}(I_n))_{n\ge 1}$ is saturated and consists of squarefree monomial ideals for a suitable monomial order $\le$ on $R$, then Proposition \ref{c=1_constant} implies that $\reg(I_n)$ is eventually constant, verifying Conjecture \ref{conj} in this special case. 
\end{remark}

We now consider invariant ideals generated by one monomial orbit. In order to verify Conjecture \ref{con_c=1} for such ideals we recall a basic fact on regularity of Cohen-Macaulay modules. By Hilbert's theorem (see, e.g., \cite[Corollary 4.1.8]{BH}), the Hilbert series of a nonzero finitely generated graded $R_n$-module $M$ of dimension $d$ can be uniquely written in the form
\[
 H_M(t)=\frac{Q_M(t)}{(1-t)^d}\quad\text{with }\ Q_M(t)\in\Z[t,t^{-1}]\ \text{ and }\ Q_M(1)\ne 0.
\]
When $M$ is a Cohen-Macaulay module, the coefficients of $Q_M(t)$ are non-negative  (see \cite[Corollary 4.1.10]{BH}). Moreover, it follows from \cite[Theorem 4.4.3(c)]{BH} and \cite[Corollary 4.8]{Ei05}) that
\begin{equation}
 \label{eq414}
 \reg(M)=\deg Q_M(t).
\end{equation}

Note that equivariant Hilbert series of ideals generated by one monomial orbit have been studied in \cite{GN}. The next result employs the induction method used in that paper.

\begin{prop}
\label{mon_orbit} 
Let $c = 1$ and consider a monomial $u = x_{\mu_1}^{a_1} \cdots x_{\mu_d}^{a_d} \in R$,
where $\mu_1 < \cdots< \mu_d$ and $a_1,\dots,a_d \in \N$. Let  $\Icc=(I_n)_{n\ge 1}$ be the 
$\Inc(\N)$-invariant chain with $I_n = \langle \Inc(\N)_{\mu_d, n} (u) \rangle$ if $n \ge \mu_d$. Then  
$\omega (\Icc) = \max\{ a_1,\ldots,a_d\}$ and 
\[
\reg (I_n) =  (\omega (\Icc) -1)(n- \mu_d) + a_1 + \cdots + a_d\quad\text{for all }\ n \ge \mu_d.
\]
\end{prop}

\begin{proof}
 It is evident that $\omega (\Icc) = \max\{ a_1,\ldots,a_d\}$. To prove the remaining assertion we use induction on $d \ge 1$. If $d=1$, then $I_n = \langle x_{\mu_1}^{a_1},\ldots,x_{n}^{a_1}\rangle$ for $n \ge \mu_1$. Note that the generators of $I_n$ form an $R_n$-regular sequence. So using Lemma \ref{reg mod regular element} repeatedly one gets
\[
\reg (I_n) = \reg(R_n/I_n)+1=(a_1 -1)(n-\mu_1+1)+1= (a_1 -1)(n-\mu_1)+a_1
\]
for all $n \ge \mu_1$, as claimed.
Let $d \ge 2$. Consider the monomial $v = x_{\mu_1}^{a_1} \cdots x_{\mu_{d-1}}^{a_{d-1}}$ and the chain $\Jcc=(J_n)_{n\ge 1}$ with
 \[
 J_n= \langle \Inc(\N)_{\mu_{d-1}, n} (v)\rangle\quad\text{if }\ n\ge \mu_{d-1}.
 \]
 Then for $n \ge \mu_d$ we have the following short exact sequence
\[
 0 \to (R_n/\langle J_{n-\nu_d}\rangle_{R_n}) (-a_d) \to R_n/I_n \to  R_n/ \langle I_{n-1}, x_n^{a_d}\rangle_{R_n} \to 0,
\]
where $\nu_d = \mu_d - \mu_{d-1}$, and moreover, all the nonzero modules in this sequence are Cohen-Macaulay of dimension $\mu_d -1$; see the proof of \cite[Corollary 2.2]{GN}. So the additivity of Hilbert series gives
\[
 Q_{ R_n/I_n}(t)=t^{a_d}Q_{ R_n/\langle J_{n-\nu_d}\rangle}(t)+Q_{ R_n/\langle I_{n-1}, x_n^{a_d}\rangle}(t).
\]
Since the coefficients of the polynomials in this equation are non-negative, one gets
\[
 \deg Q_{ R_n/I_n}(t)=\max\{\deg Q_{ R_n/\langle J_{n-\nu_d}\rangle}(t)+a_d,\; \deg Q_{ R_n/\langle I_{n-1}, x_n^{a_d}\rangle}(t)\}.
\]
Thus, it follows from Equation \eqref{eq414} that
\[
 \reg (R_n/I_n) = \max \{\reg (R_n/\langle J_{n-\nu_d}\rangle_{R_n}) +a_d,\ \reg (R_n/ \langle I_{n-1}, x_n^{a_d}\rangle_{R_n})\}.
\]
This together with Lemma \ref{reg mod regular element} gives
\begin{equation}
      \label{cor:recursion}
 \reg (I_n) = \max \{\reg (J_{n-\nu_d}) + a_d,\ \reg (I_{n-1}) + a_d -1\}.
\end{equation}

Now we prove the desired assertion by a second induction on $n \ge \mu_d$. If $n = \mu_d$, then $I_n = \langle x_{\mu_1}^{a_1} \cdots x_{\mu_d}^{a_d}\rangle$, and so again by Lemma \ref{reg mod regular element},
$$
\reg (I_n) = \reg(R_n/I_n)+1= a_1 + \cdots + a_d.
$$

Consider the case $n > \mu_d$.
Set $\omega' =\omega (\Jcc)= \max\{ a_1,\ldots,a_{d-1}\}$ and $\omega =\omega (\Icc)$. Then $\omega = \max \{a_d, \omega'\}$. Notice that $n - \nu_d \ge \mu_{d-1}$ is equivalent to $n \ge \mu_d$. Thus, the induction hypothesis on $d$ gives
\begin{align*}
\reg (J_{n-\nu_d}) & = (\omega'-1)(n - \nu_d - \mu_{d-1})  + a_1 + \cdots + a_{d-1} \\
& =  (\omega'-1)(n - \mu_{d}) + a_1 + \cdots + a_{d-1}
\end{align*}
for $n \ge \mu_d$.
Hence, Equation~\eqref{cor:recursion} and the induction hypothesis on $n \ge \mu_d$ yield
\begin{align*}
\reg (I_n)  = \max \{ &(\omega'-1)(n - \mu_{d})  + a_1 + \cdots + a_{d-1} + a_d , \\
 &(\omega-1)(n - \mu_d-1) +  a_1 + \cdots + a_d + a_d  -1\}\\
 = \max \{ &(\omega'-1)(n - \mu_{d}),\  (\omega-1)(n - \mu_d) - \omega + a_d\} + a_1 + \cdots + a_d.
\end{align*}
If $\omega=a_d$, then the second term gives the maximum. If $\omega>a_d$, then $\omega = \omega'$, and the first term determines the maximum. This concludes the proof.
\end{proof}

Note that Proposition \ref{mon_orbit} can be extended to the case where $u$ is a monomial in $R$ with $c>1$, providing evidence for Conjecture \ref{conj}. We leave the details to the interested reader.

\begin{remark}
Conjecture \ref{con_c=1} has recently been confirmed for any $\Sym(\infty)$-invariant chain of monomial ideals by Murai \cite{Mu}, who employs combinatorial methods to study the asymptotic behavior of Betti tables of ideals in such a chain. In a recent talk, Raicu also announced that he independently obtained the same result by a different approach. 
\end{remark}


\section{Proof of Theorem \ref{thm_bounding_reg_monomial}}\label{sec5}

This section is devoted to the proof of Theorem \ref{thm_bounding_reg_monomial}. As pointed out in Remark \ref{rem: equi}, we only need to prove Corollary \ref{cor_i_bound}. The proof of the latter result is divided into two steps. In the first step we prove a slightly more general result, showing that Corollary \ref{cor_i_bound} holds for any bound that satisfies certain conditions. The fact that those conditions are fulfilled by the bound given in Corollary \ref{cor_i_bound} is shown in the second step.

Let us begin with some preparations. The following recursive formula for the regularity of a monomial ideal, which extends \cite[Lemma 2.10]{DHS}, was proved in \cite[Corollary 3.3 and Theorem 4.7]{CH+}.

\begin{lem}
	\label{cor:bdl}
	If $S$ is a Noetherian polynomial ring over $K$, $x$ any variable of $S$, and $J \subseteq S$ any nonzero monomial ideal, then one has
	\[
	\max\{\reg (J : x),\ \reg (J, x) \} \le 
	\reg (J) \in \{\reg (J : x) +  1,\ \reg (J, x)\}.
	\]
\end{lem}

Recall that the lower bound on $\reg (J)$ in the previous lemma was used in the proof of Proposition \ref{c=1_constant}. For our purpose below, it suffices to consider the containment statement of the lemma. Repeatedly applying this statement, we obtain:

\begin{cor}
	\label{cor:repeat bdl}
	Let $S$ be a Noetherian polynomial ring over $K$, $x$ a variable of $S$, and $J \subseteq S$ a nonzero monomial ideal. Let $d\ge 0$ be an integer such that $J:x^d=J:x^{d+1}$. Then 
	\[
	\reg (J) \in \{\reg (J:x^k,x)+k \mid 0\le k \le d \}.
	\]
\end{cor}

\begin{proof}
	Using Lemma \ref{cor:bdl} we get
	\[
	\reg (J : x^k) \in \{\reg (J : x^{k+1}) +  1,\ \reg (J : x^k, x)\}
	\]
	for every integer $k \ge 0$. By assumption, $x$ is a non-zero-divisor on $S/(J: x^d)$. So according to Lemma \ref{reg mod regular element},
	\[
	\reg (J : x^d) = \reg (J : x^d, x).
	\]
    Combining this with the first estimate for $k = 0,\ldots,d-1$, the result follows.
\end{proof}

In the next lemma we apply the previous result to $\Inc(\N)^i$-invariant chains of monomial ideals, for which we need some more notation. Let $\Icc=(I_n)_{n\ge 1}$ be a nonzero $\Inc(\N)^i$-invariant chain of monomial ideals with $r=\ind^i(\Icc)$. We write $\delta(I_r)$ for the maximal degree of a monomial in $G(I_r)$ and set
\[
q(\Icc)=\sum_{j=0}^{\delta(I_r)}\dim_K (R_r/I_r)_j.
\]
Moreover, let $\sigma_i\in\Inc(\N)^i$ denote the \emph{$i$-shift} defined as follows:
\[
\sigma_i(j) =
\begin{cases}
j &\text{if }\ 1\le j\le i,\\
j+1 &\text{if }\ j\ge i+1.
\end{cases}
\]

\begin{lem}
	\label{lem_linearform_manysteps}
	Let $\Icc=(I_n)_{n\ge 1}$ be a nonzero $\Inc(\N)^i$-invariant chain of monomial ideals with $r=\ind^i(\Icc)$. 
	For each ${\bf e}\in \mathbb{Z}^c_{\ge0}$ consider a chain of monomial ideals $\Icc_{\bf e}=(I_{{\bf e},n})_{n\ge 1}$
	given by
	\[
	I_{{\bf e},n}=\begin{cases}
	0 &\text{if }\ 1\le n\le r,\\
	\langle I_n:x_{1,i+1}^{e_1}\cdots x_{c,i+1}^{e_c},x_{1,i+1},\ldots,x_{c,i+1}\rangle &\text{if }\  n\ge r+1.
	\end{cases}
	\]
	Fix an integer $n_0 \ge r+1$. 
	Set
	\[
	\begin{aligned}
	d_k&= \max\{e\ge 0\mid x_{k,i+1}^e ~ \text{divides some monomial in $G(I_{n_0})$}\}\quad \text{for } k\in [c],\\
	E&=\{{\bf e}=(e_1,\ldots,e_c)\in \mathbb{Z}^c\mid 0\le e_k\le d_k\ \text{ for all } k\in [c]\}.
	\end{aligned}
	\]
Then the following statements hold:
	\begin{enumerate}
	 \item
	 $\Icc_{\bf e}$ is an $\Inc(\N)^{i+1}$-invariant chain with $\ind^{i+1}(\Icc_{\bf e})= r+1$.
	 \item
	 We have
	\[
	\reg {(I_n)} \in \{\reg (I_{{\bf e},n})+|{\bf e}|\mid {\bf e}\in E\} \quad \text{for all }\ n\ge n_0,
	\]
	where $|{\bf e}|=e_1+\cdots+e_c$.
        \item
	We have
	\[
	 q(\Icc_{\bf e})\le q(\Icc)\quad\text{for all }\ {\bf e}\in \mathbb{Z}^c_{\ge0},
	\]
	and if the equality holds, then
	\[
	I_{{\bf e},n+1}=\langle\sigma_i(I_n),x_{1,i+1},\ldots,x_{c,i+1}\rangle\ \text{ and }\  \reg(I_{{\bf e},n+1})= \reg(I_n)\quad \text{for all }\ n\ge r.
	\]
	\end{enumerate}
\end{lem}

\begin{proof}
	(i) follows from \cite[Lemma 6.10]{NR17}. Let us prove (ii). As in the proof of Lemma \ref{lem_decreasing} (see \eqref{eq31}) one can show that
	\[ 
	 G(I_{n})\subseteq \Inc(\N)^i_{n_0,n}(G(I_{n_0}))\quad \text{for all }\ n\ge n_0.
	\]
        Since $x_{k,i+1}^{d_k+1}$ does not divide any monomial in $G(I_{n_0})$, the containment above implies that $x_{k,i+1}^{d_k+1}$ also does not divide any monomial in $G(I_{n})$ for all $n\ge n_0$. Hence,
	\[
	\begin{aligned}
	&\langle I_n:x_{1,i+1}^{e_1}\cdots x_{k-1,i+1}^{e_{k-1}},x_{1,i+1},\ldots,x_{k-1,i+1}\rangle:x_{k,i+1}^{d_k}\\
	&=\langle I_n:x_{1,i+1}^{e_1}\cdots x_{k-1,i+1}^{e_{k-1}},x_{1,i+1},\ldots,x_{k-1,i+1}\rangle:x_{k,i+1}^{d_k+1}
	\end{aligned}
	\]
	for all $n\ge n_0$, $1\le k\le c$ and $0\le e_1\le d_1,\dots,0\le e_{k-1}\le d_{k-1}$. Now using Corollary \ref{cor:repeat bdl}, the desired assertion follows by induction.
		
	Finally, we prove (iii). Note that the estimate $q(\Icc_{\bf e})\le q(\Icc)$ is shown in the proof of \cite[Theorem 6.2]{NR17} (page 227). To finish the proof we employ the following implications:
	\begin{align*}
	 q(\Icc_{\bf e})= q(\Icc) &\Rightarrow I_{{\bf e},r+1}=\langle\sigma_i(I_r),x_{1,i+1},\ldots,x_{c,i+1}\rangle\\
	 &\Rightarrow I_{{\bf e},n+1}=\langle\sigma_i(I_n),x_{1,i+1},\ldots,x_{c,i+1}\rangle\quad \text{for all $n\ge r$}\\
	 &\Rightarrow \reg(I_{{\bf e},n+1})= \reg(I_n) \quad \text{for all $n\ge r$}.
	\end{align*}
        The first two implications are shown in the proofs of \cite[Theorem 6.2]{NR17} and \cite[Lemma 6.11]{NR17}, respectively. It remains to check the last one. By definition of $\sigma_i$, every monomial in $G(\sigma_i(I_n))$ does not involve $x_{k,i+1}$ for all $k\in[c]$. Hence, $x_{1,i+1},\ldots,x_{c,i+1}$ form a regular sequence on $R_{n+1}/\langle\sigma_i(I_n)\rangle$. So it follows from Lemma \ref{reg mod regular element} that
        \[
         \reg(\langle\sigma_i(I_n),x_{1,i+1},\ldots,x_{c,i+1}\rangle)=\reg(\langle\sigma_i(I_n)\rangle)= \reg(I_n),
        \]
        where the last equality is obvious, since the action of $\sigma_i$ is merely a shift of variables.
\end{proof}

Next we prove a version of Corollary \ref{cor_i_bound}, stating that this result holds for more general bounds. Let us first make precise what is meant by \textquotedblleft more general bounds". 

\begin{defn}
        \label{BC}
	Let $\mathcal{F}$ denote the family of all pairs $(i,\Icc)$, where $i\ge 0$ is an integer and $\Icc=(I_n)_{n\ge 1}$ is a nonzero $\Inc(\N)^i$-invariant chain of monomial ideals. Associate to each $(i,\Icc)\in \mathcal{F}$ a real number $C^{(i)}(\Icc)\ge0$. 
	Then the family $\{C^{(i)}(\Icc)\}$ is said to be a family of {\em bounding coefficients} if, for each pair $(i,\Icc) \in \mathcal{F}$, there is an integer $n_0 \ge \ind^i (\Icc)+1$ such that the following conditions are satisfied for every  ${\bf e}\in E$:
	\begin{enumerate}[(BC1)]
		\item
		if $q(\Icc_{{\bf e}})< q(\Icc)$, then $C^{(i+1)}(\Icc_{\bf e})\le C^{(i)}(\Icc)$,
		\item
		if  $q(\Icc_{{\bf e}})= q(\Icc)$, then $|{\bf e}|\le C^{(i)}(\Icc)$,
	\end{enumerate}
	where the set $E$ and the chain $\Icc_{\bf e}=(I_{{\bf e},n})_{n\ge 1}$ are defined as in Lemma \ref{lem_linearform_manysteps}.
\end{defn}

\begin{thm}
 \label{bounding_coefficients}
	Let $\Icc=(I_n)_{n\ge 1}$ be a nonzero $\Inc(\N)^i$-invariant chain of monomial ideals. Then for any family of bounding coefficients $\{C^{(i)}(\Icc)\}$, there is a constant $D^{(i)}(\Icc)$ such that
	\[
	\reg(I_n)\le C^{(i)}(\Icc)n+D^{(i)}(\Icc)\quad \text{for }\ n\gg 0.
	\]
\end{thm}

\begin{proof}
Let us fix a family of bounding coefficients $\{C^{(i)}(\Icc)\}$. Following the idea of the proof of \cite[Theorem 6.2]{NR17}, we argue by induction on $q(\Icc)$. Set $r=\ind^i(\Icc)$. If $q(\Icc)=0$, then $I_r=R_r$, and so $I_n=R_n$ for every $n\ge r$. It follows that $\reg(I_n)=0$ for $n\ge r$, and we are done by choosing $D^{(i)}(\Icc)=0$.

Now assume $q(\Icc)\ge 1$. For each ${\bf e}\in \mathbb{Z}^c_{\ge0}$,
consider the chain $\Icc_{\bf e}=(I_{{\bf e},n})_{n\ge 1}$ as in Lemma \ref{lem_linearform_manysteps}.
Set
\[
\begin{aligned}
d_k&= \max\{e\ge 0\mid x_{k,i+1}^e ~ \text{divides some monomial in $G(I_{n_0})$}\}\ \text { for  $ k\in [c]$},\\
E&=\{{\bf e}=(e_1,\ldots,e_c)\in \mathbb{Z}^c\mid 0\le e_k\le d_k \ \text { for all  $ k\in [c]$}\},
\end{aligned}
\]
where $n_0 \ge r+1$ is a suitable integer such that Conditions (BC1) and (BC2) in Definition \ref{BC} are satisfied. According to Lemma \ref{lem_linearform_manysteps}, $\Icc_{\bf e}$ is an $\Inc(\N)^{i+1}$-invariant chain with 
\[
q(\Icc_{{\bf e}})\le q(\Icc).
\]
Note that $|{\bf e}|=e_1+\cdots+e_c\le d:=\sum_{k=1}^cd_k$ for all ${\bf e}\in E$. We write $E=E_1\cup E_2$ with
\[
E_1=\{ {\bf e}\in E\mid q(\Icc_{{\bf e}})< q(\Icc)\}\quad \text{and}\quad
E_2=\{ {\bf e}\in E\mid q(\Icc_{{\bf e}})= q(\Icc)\}.
\]

If ${\bf e}\in E_1$, then by the induction hypothesis applied to $(i+1,\Icc_{\bf e})\in\mathcal{F}$, there exist numbers $D^{(i+1)}(\Icc_{\bf e})$ and $N(\Icc_{\bf e})$ such that
\[
\reg (I_{{\bf e},n}) \le C^{(i+1)}(\Icc_{\bf e}) n +D^{(i+1)}(\Icc_{\bf e})\quad \text{for all }\ n\ge N(\Icc_{\bf e}).
\]
Since $C^{(i+1)}(\Icc_{\bf e})\le C^{(i)}(\Icc)$ by Condition (BC1), we get
\[
\reg (I_{{\bf e},n}) \le C^{(i)}(\Icc) n +D_1\quad \text{for all }\ {\bf e}\in E_1\ \text{ and }\ n\ge N_1,
\]
where $D_1=\max\{D^{(i+1)}(\Icc_{\bf e})\mid {\bf e}\in E_1\}$ and $N_1=\max\{N(\Icc_{\bf e})\mid {\bf e}\in E_1\}$.

On the other hand, if ${\bf e}\in E_2$, then Lemma \ref{lem_linearform_manysteps}(iii) gives
\[
\reg (I_{{\bf e},n})=\reg (I_{n-1}) \quad \text{for all }\ n\ge r+1.
\]
Moreover, one has
$
|{\bf e}|\le C^{(i)}(\Icc)
$
for all ${\bf e}\in E_2$ by Condition (BC2).

Now for  $n> N=\max\{N_1,n_0\}$, it follows from Lemma \ref{lem_linearform_manysteps}(ii) that
\begin{align*}
\reg (I_n) &\le \max\{\reg (I_{{\bf e},n})+|{\bf e}|\mid {\bf e}\in E\}  \\
               &= \max\big\{\max\{\reg (I_{{\bf e},n})+|{\bf e}|\mid {\bf e}\in E_1\},\ \max\{\reg (I_{{\bf e},n})+|{\bf e}|\mid {\bf e}\in E_2\}\big\}  \\
               & \le \max \{C^{(i)}(\Icc) n +D_1+d,\ \reg(I_{n-1})+C^{(i)}(\Icc)\}\\
               &\le C^{(i)}(\Icc) n +D^{(i)}(\Icc),
\end{align*}
where
$
 D^{(i)}(\Icc)=\max\{D_1+d,\ \reg (I_N)-C^{(i)}(\Icc)N\}.
$
This concludes the proof.
\end{proof}

To complete the proof of Theorem \ref{thm_bounding_reg_monomial} and Corollary \ref{cor_i_bound}, in the remaining part of this section we will show:

\begin{prop}
 \label{i-weight BC}
 Let $\mathcal{F}$ be the family of all pairs $(i,\Icc)$ as in Definition \ref{BC}. For each $(i,\Icc)\in\mathcal{F}$, set 
        \[
	C^{(i)}(\Icc)=\max\{\omega^{(i)}(\Icc)-1,0\}+\max\Big\{\sum_{k\ne l}w_k^{(i)}(\Icc)\mid l\in[c]\Big\}. 
	\] 
 Then $\{C^{(i)}(\Icc)\}$ is a family of bounding coefficients.
\end{prop}

The proof of this result requires further preparations. First of all, it is convenient to introduce one more weight function.

\begin{defn}
 Let $I_n\subseteq R_n$ be a nonzero monomial ideal. Set
		\[
		\widetilde{\omega}^{(i)}(I_n) =		
		\min \{  w^{(i)}(u) \mid u \in G_+^{(i)}(I_n)\cup G^{(i)}(I_n)\}.
		\]
We adopt the convention that $\widetilde{\omega}^{(i)}(I_n) =0$ in case $G_+^{(i)}(I_n)\cup G^{(i)}(I_n)=\emptyset$.
		
A monomial $u \in G(I_n)$ is called \emph{$i$-critical} if $w^{(i)}(u) = \widetilde{\omega}^{(i)}(I_n)$ and $\deg u \le \deg u'$ for every $u' \in G(I_n)$ with $w^{(i)} (u') = \widetilde{\omega}^{(i)} (I_n)$. 
\end{defn}

\begin{example}
 Consider the ideal $I_6$ in Example \ref{ex33}. We have
 \[
  \widetilde{\omega}^{(2)}(I_6) =\min \{w^{(2)}(u_2),\; w^{(2)}(u_3),\; w^{(2)}(u_4),\; w^{(2)}(u_5)\}=2.
 \]
 Note that $\widetilde{\omega}^{(2)}(I_6) =w^{(2)}(u_2)= w^{(2)}(u_3)$, but only $u_2$ is a 2-critical monomial of $I_6$ since $\deg u_2<\deg u_3$.
\end{example}

Comparing the definitions of ${\omega}^{(i)}(I_n)$ and $\widetilde{\omega}^{(i)}(I_n)$ one sees immediately that 
\begin{equation}
\label{eq_weights}
 \widetilde{\omega}^{(i)}(I_n)\le {\omega}^{(i)}(I_n). 
\end{equation}
Moreover, the weight $\widetilde{\omega}^{(i)}$ can also be extended to invariant chains.

\begin{lem}
 \label{prop:i-weight properties}
 Let $\Icc=(I_n)_{n\ge 1}$ be a nonzero $\Inc(\N)^i$-invariant chain of monomial ideals with $r= \ind^i(\Icc)$.
 Then the following statements hold:
 \begin{enumerate}
 	\item
 	Every $i$-critical monomial of $I_r$ is a minimal generator of $I_{n}$ for all $n\ge r$.
 	\item
 	One has
 	\[
 	 \widetilde{\omega}^{(i)}(I_n)=\widetilde{\omega}^{(i)}(I_r)\quad \text{for all }\ n\ge r,
 	\]
 	and this common value will be denoted by $\widetilde{\omega}^{(i)}(\Icc)$.
\end{enumerate}
\end{lem}

\begin{proof}
(i) Let $u$ be an $i$-critical monomial of $I_r$. Then
	\[
	w^{(i)}(u)=\widetilde{\omega}^{(i)}(I_r) =		
	\min \{w^{(i)}(u') \mid u' \in G_+^{(i)}(I_r)\cup G^{(i)}(I_r)\}
	\]
	and $u$ has lowest degree among all monomials of $G(I_r)$ with this property. Assume that $u\notin G(I_{n})$ for some $n\ge r$. Since $u\in I_r\subseteq I_n$, $u$ must be divisible by some $v\in G(I_{n})$ with $\deg v < \deg u$. As in the proof of Lemma \ref{lem_decreasing} one has
	\begin{equation}
	\label{eq61}
	 G(I_{n})\subseteq \Inc (\N)^i_{r, n}(G(I_r)).
	\end{equation}
	Hence there exist $u'\in G(I_r)$ and $\pi \in \Inc (\N)^i_{r, n}$ such that $v=\pi (u')$. We show that 
	\[
	u' \in G_+^{(i)}(I_r)\cup G^{(i)}(I_r). 
	\]
        Indeed, if $u' \in G_-^{(i)}(I_r)$, then $u'$ is fixed under the action of $\Inc (\N)^i$ and $u'=\pi (u')=v$ divides $u$. This contradicts the fact that $u\in G(I_r)$. Thus, $u'\in G_+^{(i)}(I_r)\cup G^{(i)}(I_r)$. We have
	 \[
	 w^{(i)}(u') = w^{(i)}(\pi (u'))= w^{(i)}(v)\le w^{(i)}(u) = \widetilde{\omega}^{(i)}(I_r).
	 \]
        By definition of $\widetilde{\omega}^{(i)}(I_r)$, this implies $w^{(i)}(u') = \widetilde{\omega}^{(i)}(I_r)$. So $\deg u'=\deg v < \deg u$ is a contradiction to the choice of $u$ as an $i$-critical monomial of $I_r$.

(ii) Let $n\ge r$. We may assume $G_+^{(i)}(I_r)\cup G^{(i)}(I_r)\ne \emptyset$, since otherwise $G(I_r)=G_-^{(i)}(I_r)$ is fixed by $\Inc (\N)^i$, implying $G(I_n)=G(I_r)$ and thus the assertion holds trivially. Let $u$ be an $i$-critical monomial of $I_r$. Then $u\in G_+^{(i)}(I_r)\cup G^{(i)}(I_r)$. According to (i), $u\in G(I_n)$. Hence, $u\in G_+^{(i)}(I_n)\cup G^{(i)}(I_n)$. This yields
\[
 \widetilde{\omega}^{(i)}(I_r)=w^{(i)}(u)\ge \widetilde{\omega}^{(i)}(I_n).
\]
On the other hand, it follows from \eqref{eq61} that
\[
 G_+^{(i)}(I_{n})\subseteq \Inc (\N)^i_{r, n}(G_+^{(i)}(I_r))\quad\text{and}\quad G^{(i)}(I_{n})\subseteq \Inc (\N)^i_{r, n}(G^{(i)}(I_r)).
\]
Now arguing as in the proof of Lemma \ref{lem_decreasing} (see \eqref{eq32}) one can show that
\[
 \{w^{(i)}(v)\mid v\in G_+^{(i)}(I_n)\cup G^{(i)}(I_{n})\}\subseteq \{w^{(i)}(v')\mid v'\in G_+^{(i)}(I_r)\cup G^{(i)}(I_{r})\}.
\]
This implies
\[
\begin{aligned}
   \widetilde{\omega}^{(i)}(I_n)&=\min\{w^{(i)}(v)\mid v\in G_+^{(i)}(I_n)\cup G^{(i)}(I_{n})\}\\
   &\ge \min\{w^{(i)}(v')\mid v'\in G_+^{(i)}(I_r)\cup G^{(i)}(I_{r})\}=\widetilde{\omega}^{(i)}(I_r),
  \end{aligned}
\]
which concludes the proof.
\end{proof}

The proof of Proposition \ref{i-weight BC} is essentially based on the following facts.

\begin{lem}
	\label{compare weights}
	Let $\Icc=(I_n)_{n\ge 1}$ be a nonzero $\Inc(\N)^i$-invariant chain of monomial ideals.
	For each ${\bf e}=(e_1,\dots,e_c)\in \mathbb{Z}^c_{\ge 0}$, consider the chain $\Icc_{\bf e}=(I_{{\bf e},n})_{n\ge 1}$ as in Lemma \ref{lem_linearform_manysteps}.
	Then the following statements hold:
	\begin{enumerate}
	\item
	For every $k\in[c]$ and every ${\bf e}\in \mathbb{Z}^c_{\ge 0}$ one has
	\[
	 w_k^{(i+1)}(\Icc_{\bf e})\le w_k^{(i)}(\Icc) \quad \text{and}\quad
	 \omega^{(i+1)}(\Icc_{\bf e})\le \omega^{(i)}(\Icc).
	\]
        \item
	If $\widetilde{\omega}^{(i)}(\Icc)> 0$ and $e_k\ge \widetilde{\omega}^{(i)}(\Icc)$ for all $k\in [c]$, then $q(\Icc_{\bf e})< q(\Icc)$.
	\end{enumerate}
\end{lem}

\begin{proof} 
	 (i) Set $r= \ind^i(\Icc)$. By definition, it suffices to prove that
	 \[
	 w_k^{(i+1)}(I_{{\bf e},n})\le w_k^{(i)}(I_n) \quad \text{and}\quad
	 \omega^{(i+1)}(I_{{\bf e},n+1})\le \omega^{(i)}(I_n)\quad\text{for all }\ n\ge r+1.
	\]
	
	 For the estimate $w_k^{(i+1)}(I_{{\bf e},n})\le w_k^{(i)}(I_n)$ we only need to show that for each $u\in G(I_{{\bf e},n})$ there exists $v\in G(I_n)$ such that $w_k^{(i+1)}(u)\le w_k^{(i)}(v)$. Evidently, if $u\in G(I_{{\bf e},n})$, then either $u=x_{l,i+1}$ for some $l\in [c]$, or $u$ divides some $v\in G(I_{n})$. The former case is trivial since $w_k^{(i+1)}(x_{l,i+1})=0$. In the latter case, we have
	\[
	 w_k^{(i+1)}(u)\le w_k^{(i+1)}(v)\le w_k^{(i)}(v),
	\]
	as desired.
	
	Now we show that $\omega^{(i+1)}(I_{{\bf e},n+1})\le \omega^{(i)}(I_n)$ for all $n\ge r+1$. If $I_{{\bf e},n+1} =R_{n+1}$, then $\omega^{(i+1)}(I_{{\bf e},n+1})=\omega^{(i+1)}(R_{n+1})=0$, and we are done. Assume that $I_{{\bf e},n+1} \neq R_{n+1}$. We distinguish two cases:
	
	\emph{Case 1}: $G^{(i)}_+(I_n)\ne\emptyset$. Then $\omega^{(i)}(I_n)=\min\{w^{(i)}(u)\mid u\in G^{(i)}_+(I_n)\}$. Choose $u\in G^{(i)}_+(I_n)$ such that $w^{(i)}(u)=\omega^{(i)}(I_n)$. Let $\sigma_i$ be the $i$-shift introduced above Lemma \ref{lem_linearform_manysteps}. 
	Since $\sigma_i(u)\in I_{n+1}\subseteq I_{{\bf e},n+1}$, it is divisible by some $v\in G(I_{{\bf e},n+1})$. As $v\ne 1$, one has
	\[
	 \min(v)\ge\min(\sigma_i(u))=\min(u)+1>i+1.
	\]
	Thus $v\in G^{(i+1)}_+(I_{{\bf e},n+1})$, and the latter is non-empty. This implies 
	\[
	 \omega^{(i+1)}(I_{{\bf e},n+1})\le w^{(i+1)}(v)\le w^{(i)}(v)\le w^{(i)}(\sigma_i(u))=w^{(i)}(u)=\omega^{(i)}(I_n).
	\]
	
	\emph{Case 2}: $G^{(i)}_+(I_n)=\emptyset$. Using Lemma \ref{lem_compare} and the estimate $w_k^{(i+1)}(I_{{\bf e},n+1})\le w_k^{(i)}(I_{n+1})$ shown above we get
	\[
	 \begin{aligned}
	  \omega^{(i+1)}(I_{{\bf e},n+1})&\le \max\{w_k^{(i+1)}(I_{{\bf e},n+1})\mid k\in[c]\}\le\max\{w_k^{(i)}(I_{n+1})\mid k\in[c]\}\\
	  &\le \max\{w_k^{(i)}(I_{n})\mid k\in[c]\}=\omega^{(i)}(I_n),
	 \end{aligned}
	\]
	where the third inequality is due to Lemma \ref{lem_decreasing}.

	(ii)  Set $\widetilde{\omega}=\widetilde{\omega}^{(i)}(\Icc)=\widetilde{\omega}^{(i)}(I_r)$. According to Lemma \ref{lem_linearform_manysteps}(iii), it is enough to show that if $e_k\ge \widetilde{\omega}>0$ for all $k\in [c]$, then
	\[
	 \langle\sigma_i(I_r),x_{1,i+1},\ldots,x_{c,i+1}\rangle\subsetneq I_{{\bf e},r+1}.
	\]
	Recall that, as $\widetilde{\omega}>0$, each $i$-critical monomial of $I_r$ is divisible by $x_{k,j}^{\widetilde{\omega}}$ for some $j > i$ and $k\in [c]$.
	Let $j_0$ be the smallest index $j$ with $j > i$ such that $x_{k,j}^{\widetilde{\omega}}$ divides an $i$-critical monomial of $I_r$ for some $k\in [c]$.
	An $i$-critical monomial of $I_r$ is said to be \emph{$i$-distinguished} if it is divisible by $x_{k,j_0}^{\widetilde{\omega}}$ for some  $k\in [c]$.
	We consider two cases:
	
    \emph{Case 1}: $I_r$ has an $i$-distinguished monomial $u$ that is divisible by $x_{k,i+1}$ for some $k\in [c]$. Then we may write $u=x_{1,i+1}^{e_1'}\cdots x_{c,i+1}^{e_c'}v$ with $v$ not divisible by $x_{l,i+1}$ for any $l\in [c]$. Evidently, $e_k'\ge 1$ and $ e_l'\le w_l^{(i)}(u)\le w^{(i)}(u) = \widetilde{\omega}$ for all $l\in [c]$. So from $e_1,\dots,e_c \ge \widetilde{\omega}$ it follows that
    \[
       v = {u}/x_{1,i+1}^{e_1'}\cdots x_{c,i+1}^{e_c'} \in I_{r+1} : x_{1,i+1}^{e_1}\cdots x_{c,i+1}^{e_c}.
    \]
    Since $u$ is a minimal generator of $I_{r+1}$ by Lemma \ref{prop:i-weight properties}(i), $v \notin I_{r+1}$. Thus $v \notin \sigma_i (I_r)$, and therefore,
    \[
      \langle\sigma_i(I_r),x_{1,i+1},\ldots,x_{c,i+1}\rangle\subsetneq I_{{\bf e},r+1}.
    \]

   \emph{Case 2}: $x_{k,i+1}$ does not divide any $i$-distinguished monomial of $I_r$ for every $k\in [c]$. Then $j_0 > i+1$. Let $u \in I_r$ be any $i$-distinguished monomial. If $u\in \sigma_i (I_r)$, then $u = \sigma_i (v)$ for some $v \in I_r$. Note that $w^{(i)}(v)=w^{(i)}(\sigma_i (v))=w^{(i)}(u)$ and $\deg v=\deg u$. So $v$ is also an $i$-critical monomial of $I_r$. Since $u$ is divisible by $x_{k,j_0}^{\widetilde{\omega}}$ for some $k\in [c]$, $v$ is divisible by $x_{k,j_0-1}^{\widetilde{\omega}}$. This contradicts the definition of $j_0$. Hence
   \[
   u \in I_{{\bf e},r+1}\setminus \langle\sigma_i(I_r),x_{1,i+1},\ldots,x_{c,i+1}\rangle,
   \]
   which yields the desired claim.
\end{proof}

We are now ready to prove Proposition \ref{i-weight BC}.

\begin{proof}[Proof of Proposition \ref{i-weight BC}]
 We show that the family $\{C^{(i)}(\Icc)\}$ satisfies both Conditions (BC1) and (BC2) in Definition \ref{BC}. Let $(i,\Icc)\in\mathcal{F}$. Choose $n_0\ge \ind^i(\Icc)+1$ sufficiently large such that $\omega^{(i)}(\Icc)=\omega^{(i)}(I_{n_0})$ and $w_k^{(i)}(\Icc)=w_k^{(i)}(I_{n_0})$ for all $k\in [c]$.
 Define the set $E$ and the chain $\Icc_{\bf e}=(I_{{\bf e},n})_{n\ge 1}$ as in Lemma \ref{lem_linearform_manysteps}. So
 \[
 E=\{{\bf e}=(e_1,\ldots,e_c)\in \mathbb{Z}^c\mid 0\le e_k\le d_k\ \text{ for all } k\in [c]\},
 \]
 where
 \[
  d_k= \max\{e\ge 0\mid x_{k,i+1}^e ~ \text{divides some monomial in $G(I_{n_0})$}\}.
 \]

From Lemma \ref{compare weights}(i) it follows that
\[
 C^{(i+1)}(\Icc_{\bf e})\le C^{(i)}(\Icc)\quad\text{for all } {\bf e}\in \mathbb{Z}^c_{\ge 0}.
\]
In particular, this implies that Condition (BC1) is satisfied.

Now assume ${\bf e}\in E$ and $q(\Icc_{\bf e})= q(\Icc)$. Recall that $\widetilde{\omega}^{(i)}(\Icc)=\widetilde{\omega}^{(i)}(I_{n_0})$ by Lemma \ref{prop:i-weight properties}(ii). If $\widetilde{\omega}^{(i)}(I_{n_0})=0$, then $G(I_{n_0})=G^{(i)}_-(I_{n_0})$. It follows that $d_k=0$ for all $k\in[c]$, giving $E=\{(0,\ldots,0)\}$. Hence, Condition (BC2) holds trivially. Consider the case $\widetilde{\omega}^{(i)}(I_{n_0})>0$. Then by Lemma \ref{compare weights}(ii) and Inequality \eqref{eq_weights}, there exists some $\ k\in [c]$ such that
\[
 e_k<\widetilde{\omega}^{(i)}(I_{n_0})\le {\omega}^{(i)}(I_{n_0})={\omega}^{(i)}(\Icc).
\]
Since it is clear that $d_l\le w_l^{(i)}(I_{n_0})=w_l^{(i)}(\Icc)$ for all $l\in[c]$, we get
\[
 |{\bf e}|=e_1+\cdots+e_c\le e_k+\sum_{l\ne k}d_l\le {\omega}^{(i)}(\Icc) -1+ \sum_{l\ne k}w_l^{(i)}(\Icc)\le C^{(i)}(\Icc)
\]
for all ${\bf e}\in E$ with $q(\Icc_{\bf e})= q(\Icc)$. Thus, Condition (BC2) is also satisfied. This concludes the proof.
\end{proof}

\begin{remark}
 From the above proof of Proposition \ref{i-weight BC} one might expect that Corollary \ref{cor_i_bound} still holds if the coefficient $C^{(i)}(\Icc)$ is replaced by
 \[
  \widetilde{C}^{(i)}(\Icc)=\max\{\widetilde{\omega}^{(i)}(\Icc)-1,0\}+\max\Big\{\sum_{k\ne l}w_k^{(i)}(\Icc)\mid l\in[c]\Big\}.
 \]
 However, this may not be true in general. The reason is that the inequality
 \[
  \widetilde{\omega}^{(i)}(\Icc_{\bf e})\le \widetilde{\omega}^{(i)}(\Icc)
 \]
 can be violated for some ${\bf e}\in E$, and thus the family $\{\widetilde{C}^{(i)}(\Icc)\}$ does not necessarily satisfy Condition (BC2) in Definition \ref{BC}. For a concrete example, the reader is invited to check this with the chain $\Icc'=(I_n')_{n\ge 1}$ in Example \ref{ex31}. This observation again indicates that the bound in Theorem \ref{thm_bounding_reg_monomial} is rather sharp.
\end{remark}

Finally, we conclude this section with the proof of Theorem \ref{thm_bounding_reg_monomial}.

\begin{proof}[Proof of Theorem \ref{thm_bounding_reg_monomial}]
By Remark \ref{rem: equi} it suffices to prove Corollary \ref{cor_i_bound}. But this result follows directly from Theorem \ref{bounding_coefficients} and Proposition~\ref{i-weight BC}.
\end{proof}


\section{Extremal chains}\label{sec6}

In this section we propose a method to attack Conjecture \ref{conj}, and we establish it for chains of ideals that are extremal in a certain sense. As an application, we show that the conjecture is true for chains whose ideals are eventually Artinian.

Basically, our method relies on a careful analysis of the proofs of Lemma \ref{lem_linearform_manysteps}(ii) and Theorem \ref{bounding_coefficients}.
Let $\Icc=(I_n)_{n\ge 1}$ be an $\Inc(\N)^i$-invariant chain of monomial ideals. 
Using the notation of Lemma~\ref{lem_linearform_manysteps}, it follows from the second statement of this lemma that
\[
 \reg {(I_n)} \le\max\{\reg (I_{{\bf e},n})+|{\bf e}|\mid {\bf e}\in E\}\quad \text{ for all }\ n\gg0.
\]
We are interested in chains for which the above inequality eventually becomes an equality for some subset $E$ of $\Z^c_{\ge 0}$.

\begin{defn}
 \label{extremal}
 Let $\mathcal{F}$ be the family of all pairs $(i,\Icc)$ as in Definition \ref{BC}. We say that a subfamily $\mathcal{E}\subseteq\mathcal{F}$ is \emph{extremal} if for every $(i,\Icc)\in \mathcal{E}$, there exists a subset $E(\Icc)\subset\Z^c_{\ge 0}$ such that the following conditions are satisfied:
 \begin{enumerate}[(X1)]
  \item
  $(i+1,\Icc_{\bf e})\in\mathcal{E}$ for all $\be\in E(\Icc)$,
  \item
  $\reg {(I_n)} = \max\{\reg (I_{{\bf e},n})+|{\bf e}|\mid {\bf e}\in E(\Icc)\}$ for all $n\gg0$,
 \end{enumerate}
 where the chain $\Icc_{\bf e}=(I_{{\bf e},n})_{n\ge 1}$ is defined as in Lemma \ref{lem_linearform_manysteps}. 
\end{defn}

A chain $\Icc$ is called \emph{extremal} if the pair $(i,\Icc)$ belongs to an extremal family for some integer $i\ge 0$. As shown below, chains whose ideals are eventually Artinian are extremal. We do not know whether this is true for all chains of monomial ideals. The main result of this section is the following:

\begin{thm}
 \label{thm_extremal}
 Let $\mathcal{E}$ be an extremal subfamily of $\mathcal{F}$. Then for every $(i,\Icc)\in \mathcal{E}$ with $\Icc=(I_n)_{n\ge 1}$, there exist integer constants $\mathcal{C}(\Icc),$ $\mathcal{D}(\Icc)$ such that
\[
\reg(I_n)=\mathcal{C}(\Icc)n+\mathcal{D}(\Icc) \quad \text{whenever }\ n\gg 0.
\]
\end{thm}

\begin{proof}
Following the strategy of the proof of Theorem \ref{bounding_coefficients}, we show that for every pair $(i,\Icc)\in \mathcal{E}$ there is an integer $\mathcal{C}(\Icc)$ such that
\[
\reg (I_{n+1}) = \reg (I_n) + \mathcal{C}(\Icc) \quad \text{whenever }\ n\gg 0
\]
by induction on $q(\Icc)$.
Repeating the first part of the proof of Theorem \ref{bounding_coefficients}, we may assume that $q(\Icc)\ge 1$. For each ${\bf e}\in\Z^c_{\ge 0}$ define the chain $\Icc_{\bf e}$ as in Lemma~\ref{lem_linearform_manysteps}. Since the family $\mathcal{E}$ is extremal, there exists $E(\Icc)\subset\Z^c_{\ge 0}$ such that $(i+1,\Icc_{\bf e})\in\mathcal{E}$ for all $\be\in E(\Icc)$. Lemma \ref{lem_linearform_manysteps}(iii) allows us to write $E(\Icc)=E_1\cup E_2$, where
\[
E_1=\{ {\bf e}\in E(\Icc)\mid q(\Icc_{\bf e})<q(\Icc)\}\quad \text{and}\quad
E_2=\{ {\bf e}\in E(\Icc)\mid q(\Icc_{\bf e})=q(\Icc)\}.
\]

If ${\bf e}\in E_2$, then Lemma \ref{lem_linearform_manysteps}(iii) implies
\[
\reg (I_{{\bf e},n})=\reg (I_{n-1}) \quad \text{for all }\ n\ge r+1.
\]
Hence, by Condition (X2) in Definition \ref{extremal}, there is a number $N_1\ge r+1$ such that for all $n \ge N_1$ one has
\begin{align}
    \label{eq:compar}
\reg (I_n) & = \max\big\{\max\{\reg (I_{{\bf e},n})+|{\bf e}|  \mid  {\bf e}\in E_1\},\ \max\{\reg (I_{{\bf e},n})+|{\bf e}|\mid {\bf e}\in E_2\}\big\}  \nonumber  \\
               & =  \max\big\{\max\{\reg (I_{{\bf e},n})+|{\bf e}|\mid {\bf e}\in E_1\},\ \reg (I_{n-1})+ \max\{|{\bf e}|\mid {\bf e}\in E_2\}\big\}.
\end{align}

For ${\bf e} \in E_1$, the induction hypothesis applied to $(i+1,\Icc_{\bf e})\in\mathcal{E}$ yields the existence of integers $\mathcal{C}(\Icc_{\bf e})$ and $N(\Icc_{\bf e})$ such that
\begin{equation}
    \label{eq:use IH}
\reg (I_{{\bf e}, n+1}) = \reg (I_{{\bf e}, n}) + \mathcal{C}(\Icc_{\bf e}) \quad \text{ whenever } n \ge N(\Icc_{\bf e}).
\end{equation}
Set
\[
N = \max \{ N_1,\; \max \{N(\Icc_{\bf e})  \mid {\bf e} \in E_1\}\}
\]
and
\[
\mathcal{C}(\Icc) = \max \{\max \{\mathcal{C}(\Icc_{\bf e}) \mid  {\bf e} \in E_1\},\ \max \{ |{\bf e}|  \mid {\bf e} \in E_2\} \}.
\]
We will show that
\begin{equation}
 \label{eq:induction}
 \reg (I_{n+1}) = \reg (I_n) + \mathcal{C}(\Icc) \quad \text{whenever }\ n\ge N.
\end{equation}
Indeed, if $n \ge N$ and ${\bf e}  \in E_1$, then
\begin{align*}
\reg (I_{{\bf e},n+1}) + |\be| & = \reg (I_{\be, n}) + \mathcal{C}(\Icc_{\bf e}) + |\be | \\
& \le \reg (I_n) + \mathcal{C}(\Icc_{\bf e}) &&\text{by Equation~\eqref{eq:compar} } \\
& \le \reg (I_n) + \mathcal{C}(\Icc) &&\text{by definition of $\mathcal{C}(\Icc)$. }
\end{align*}
Combined with Equation~\eqref{eq:compar}, this implies
\begin{align}
    \label{eq:upper bound}
\reg (I_{n+1}) \le \reg (I_n) + \mathcal{C}(\Icc) \quad \text{ if }\ n \ge N.
\end{align}
Moreover, Equation~\eqref{eq:compar} gives
\[
\reg (I_{n+1}) \ge \reg (I_n) + \max \{ |\be |  \mid  \be \in E_2\}  \quad \text{ if }\ n \ge N.
\]
This together with Inequality~\eqref{eq:upper bound} yields Equation~\eqref{eq:induction} if $\mathcal{C}(\Icc) = \max \{ |\be |  \mid  \be \in E_2\}$.

Thus, it remains to consider the case $\mathcal{C}(\Icc) > \max \{ |\be |  \mid \be \in E_2\}$. Suppose Equation~\eqref{eq:induction} is not true. Taking into account Inequality~\eqref{eq:upper bound}, this means that, for any $n_1 \ge N$, there is some $n > n_1$ with
\[
\reg (I_{n+1}) <  \reg (I_n)  + \mathcal{C}(\Icc).
\]
We use this to define an increasing sequence $(n_j)_{j \in \N}$ of integers:  set $n_1 = N$ and, for $j \ge 2$, let
$n_j$ be the least integer $n > n_{j-1}$ with $\reg (I_{n+1}) <  \reg (I_n)  + \mathcal{C}(\Icc)$. Thus, we obtain for every $j \ge 2$,
\begin{equation}
	\label{eq:ineq}
\reg (I_{n_j+1}) \le  \reg (I_N)  + (n_j + 1 - N) \mathcal{C}(\Icc) - (j-1).
\end{equation}
Our assumption that $\mathcal{C}(\Icc) > \max \{ |\be | \mid \be \in E_2\}$ allows us to fix some element $\be_0 \in E_1$ with $\mathcal{C}(\Icc) = \mathcal{C}(\Icc_{{\bf e}_0})$. Consider some integer $j$ with $j-1 > \reg (I_N) - \reg (I_{\be_0, N}) - | \be_0 |$, or equivalently,
\[
\reg (I_N)  + (n_j + 1 - N) \mathcal{C}(\Icc) - (j-1) < \reg (I_{\be_0,N})  + (n_j + 1 - N) \mathcal{C}(\Icc)  + | \be_0 |.
\]
Combined with Inequality~\eqref{eq:ineq}, this gives
\begin{align*}
\reg (I_{n_j + 1}) & <  \reg (I_{\be_0,N})  + (n_j + 1 - N) \mathcal{C}(\Icc)  + | \be_0 |\\
& =  \reg (I_{\be_0,N})  + (n_j + 1 - N) \mathcal{C}(\Icc_{{\bf e}_0})  + | \be_0 | \\
& = \reg (I_{\be_0, n_j+1}) + | \be_0 |  \quad \text{ by Equation~\eqref{eq:use IH}.}
\end{align*}
However, this contradicts Equation~\eqref{eq:compar}. Thus, the argument is complete.
\end{proof}

Based on the previous result we will verify Conjecture \ref{conj} for chains whose ideals are eventually Artinian.
By abuse of notation, we call an ideal $J$ of a Noetherian polynomial ring $S$ {\em Artinian} if $S/J$ is an Artinian ring. We need a more precise version of Corollary~\ref{cor:repeat bdl}.

\begin{lem}
     \label{lem:bdl for Artinian}
Let $x$ be a variable of a Noetherian polynomial ring $S$ over $K$. If $J \subseteq S$ is an Artinian monomial ideal and $d\ge 0$ is the smallest integer such that $J:x^d=J:x^{d+1}$, then
\[
\reg (J) = \max \{\reg (J : x^k, x) + k   \mid  0 \le k \le d\}.
\]
\end{lem}

\begin{proof}
If $d=0$, then $x$ is a non-zero-divisor on $S/J$. So
$\reg (J) = \reg (J, x)$ by Lemma \ref{reg mod regular element}, and we are done.
Now assume $d>0$. In this case, $J$ is a proper ideal in $S$.
Note that for any graded Artinian ideal $I \subsetneq S$ one has
\[
\reg (I) = 1 + \reg (S/I)
 = 1 + \max \{ j \in \Z  \mid  [S/I]_j \neq 0\}
\]
(see \cite[Theorem 18.4]{Pe11}). Hence, if $x \notin I$, then the exact sequence
\[
0 \to (S/(I : x))(-1) \to S/I \to S/(I, x) \to 0
\]
yields
\[
\reg (I) = \max \{1 + \reg (I : x),\ \reg (I, x) \}.
\]
Obviously, this formula also holds when $x\in I$. Now applying the formula repeatedly as in the proof of Corollary \ref{cor:repeat bdl}, the claim follows.
\end{proof}

In order to apply this result to chains of Artinian ideals, we have to refine the numbers $d_k$ and the set $E$ introduced in Lemma \ref{lem_linearform_manysteps}. Let $\Icc=(I_n)_{n\ge 1}$ be a nonzero $\Inc(\N)^i$-invariant chain of monomial ideals with $r=\ind^i(\Icc)$. Choose $n_0\ge r+1$ large enough such that $w_k^{(i)}(\Icc)=w_k^{(i)}(I_{n_0})$ for all $k\in [c]$, and for such $n_0$ define the numbers $d_k$ and the set $E$ as in Lemma \ref{lem_linearform_manysteps}. Thus, 
\[
\begin{aligned}
 d_k&= \max\{e\ge 0\mid x_{k,i+1}^e ~ \text{divides some monomial in $G(I_{n_0})$}\}\quad \text{for } k\in [c],\\
 E&=\{{\bf e}=(e_1,\dots,e_c)\in\Z^c\mid 0\le e_k\le d_k \ \text{ for all }\ k\in [c]\}.
\end{aligned}
\]

Since $w_1^{(i)}(\Icc)=w_1^{(i)}(I_{n_0})$, it is evident that $d_1$ is the largest non-negative integer such that $x_{1,i+1}^{d_1}$ divides some monomial in $G(I_n)$ for all $n\ge n_0$. Equivalently, $d_1\ge 0$ is the smallest integer such that
\[
 I_n:x_{1,i+1}^{d_1}=I_n:x_{1,i+1}^{d_1+1}\quad\text{for all } n\ge n_0.
\]

Now for each $0\le e_1\le d_1$ and $n\ge n_0$, let $d_2(n;e_1)\ge 0$ be the smallest integer such that
\[
 \langle I_n:x_{1,i+1}^{e_1},x_{1,i+1}\rangle :x_{2,i+1}^{d_2(n;e_1)}=\langle I_n:x_{1,i+1}^{e_1},x_{1,i+1}\rangle :x_{2,i+1}^{d_2(n;e_1)+1}.
\]
Evidently, $d_2(n;e_1)\le d_2$. Moreover, since $G(I_{n+1})\subseteq\Inc(\N)^i_{n,n+1}(G(I_{n}))$, we see that $d_2(n;e_1)\ge d_2(n+1;e_1)$. Thus, there exists $n_1\ge n_0$ such that $d_2(n;e_1)=d_2(n_1;e_1)$ for all $n\ge n_1$. We set $d_2(e_1)=d_2(n_1;e_1)$. 

More generally, for each $k$-tuple $(e_1,\dots,e_k)$ with $1\le k<c$, $0\le e_1\le d_1,\dots,$ $0\le e_k\le d_k(e_1,\dots,e_{k-1})$, and each $n\ge n_{k-1}$ (as above, $n_{k-1}$ is some number such that
\[
 d_k(n;e_1,\dots,e_{k-1})=d_k(e_1,\dots,e_{k-1})
\]
for all $n\ge n_{k-1}$), let $d_{k+1}(n;e_1,\dots,e_{k})\ge0$ be the smallest integer such that
\[\begin{aligned}
 \langle &I_n:x_{1,i+1}^{e_1}\cdots x_{k,i+1}^{e_k},x_{1,i+1},\dots,x_{k,i+1}\rangle :x_{k+1,i+1}^{d_{k+1}(n;e_1,\dots,e_{k})}\\
 &=\langle I_n:x_{1,i+1}^{e_1}\cdots x_{k,i+1}^{e_k},x_{1,i+1},\dots,x_{k,i+1}\rangle :x_{k+1,i+1}^{d_{k+1}(n;e_1,\dots,e_{k})+1}.
 \end{aligned}
\]
As before, $d_{k+1}(n;e_1,\dots,e_{k})$ is independent of $n$ when $n\gg0$, and for such $n$ we set
\[
d_{k+1}(e_1,\dots,e_{k})=d_{k+1}(n;e_1,\dots,e_{k}).
\]

Now define
\[
 E(\Icc)=\{{\bf e}=(e_1,\dots,e_c)\in\Z^c\mid 0\le e_k\le d_k(e_1,\dots,e_{k-1}) \ \text{ for all }\ k\in [c]\},
\]
where $d_1(\emptyset)=d_1$. Note that $d_k(e_1,\dots,e_{k-1})\le d_k$ for all $k\in [c]$.
Hence, $E(\Icc)\subseteq E$.

We have the following version of Lemma \ref{lem_linearform_manysteps}(ii) for chains of Artinian ideals:

\begin{cor}
      \label{cor:repeared artinian bdl}
 Let $\Icc=(I_n)_{n\ge 1}$ be an $\Inc(\N)^i$-invariant chain of monomial ideals with $r=\ind^i(\Icc)$. Assume $I_r\subseteq R_r$ is an Artinian ideal. Define the chain $\Icc_{\bf e}=(I_{{\bf e},n})_{n\ge 1}$ as in Lemma \ref{lem_linearform_manysteps}. Then with the set $E(\Icc)$ introduced above, we have
	\[
	\reg {(I_n)} =\max \{\reg (I_{{\bf e},n})+|{\bf e}|\mid {\bf e}\in E(\Icc)\}\quad \text{ for }\ n\gg 0.
	\]
\end{cor}

\begin{proof}
	Since $I_r$ is an Artinian ideal in $R_r$, the ideal $I_n$ is Artinian in $R_n$ for every $n \ge r$. By the above construction, for each ${\bf e}=(e_1,\dots,e_c)\in E(\Icc)$, each $k\in[c]$ and all $n\gg 0$, $d_k(e_1,\dots,e_{k-1})$ is the smallest non-negative integer such that
	\[\begin{aligned}
	\langle &I_n:x_{1,i+1}^{e_1}\cdots x_{k-1,i+1}^{e_k},x_{1,i+1},\dots,x_{k-1,i+1}\rangle :x_{k,i+1}^{d_{k}(e_1,\dots,e_{k-1})}\\
	&=\langle I_n:x_{1,i+1}^{e_1}\cdots x_{k-1,i+1}^{e_k},x_{1,i+1},\dots,x_{k-1,i+1}\rangle :x_{k,i+1}^{d_{k}(e_1,\dots,e_{k-1})+1}.
	\end{aligned}
	\]
	Using this equation together with Lemma \ref{lem:bdl for Artinian}, the claim follows by induction on $c$.
\end{proof}

Now we are ready to prove Conjecture \ref{conj} for chains of graded ideals which are eventually Artinian. Note that we do not require such ideals to be monomial ideals.

\begin{cor}
   \label{thm: artinian chains}
Let $\Icc = (I_n)_{n\ge 1}$ be an $\Inc(\N)^i$-invariant chain of graded ideals with $r = \ind^i (\Icc)$.
Suppose that $I_r$ is an Artinian ideal in $R_r$.
Then $\reg(I_n)$ is eventually a linear function in $n$, that is, there are integer constants $C,$ $D$ such that
\[
\reg(I_n)=Cn+D \quad \text{ whenever }\ n\gg 0.
\]
\end{cor}

\begin{proof}
By assumption, $I_n$ is an Artinian ideal in $R_n$ for every $n \ge r$. Let $\le$ be a monomial order on $R$ respecting $\Inc(\N)^i$. It is well-known that  the $K$-algebras $R_n/I_n$ and $R_n/\ini_{\le}(I_n)$ have the same Hilbert function. Since the regularity of an Artinian algebra is determined by its Hilbert function (see \cite[Theorem 18.4]{Pe11}), we conclude that
\[
\reg (I_n) = \reg (\ini_{\le} (I_n)) \quad \text{ whenever }\ n \ge r.
\]
So invoking Lemma \ref{lem_initial_filtration}, we may assume that $\Icc$ is a chain of monomial ideals.

With the family $\mathcal{F}$ as in Definition \ref{BC}, let $\mathcal{E}$ be the subfamily of $\mathcal{F}$ consisting of all pairs $(i',\Icc')$ such that the ideals $I_n'$ in the chain $\Icc'$ are eventually Artinian. For each such chain $\Icc'$ and each ${\bf e}\in \Z^c_{\ge0}$, define the chain $\Icc_{\bf e}'=(I_{{\bf e},n}')_{n\ge 1}$ as in Lemma \ref{lem_linearform_manysteps}. Then $I_{{\bf e},n}'$ is eventually Artinian, since $I_n'\subseteq I_{{\bf e},n}'$. It follows that $(i'+1,\Icc_{\bf e}')\in\mathcal{E}$. Together with Corollary \ref{cor:repeared artinian bdl} this implies that $\mathcal{E}$ is an extremal family. We now conclude the proof by Theorem \ref{thm_extremal}.
\end{proof}

The constant $C$ in Corollary \ref{thm: artinian chains} is certainly not always zero.

\begin{example}
 \label{artinian}
Let $I= \langle\Inc(\N) \cdot \{x_{1,1}^{a_1},\ldots,x_{c,1}^{a_c} \}\rangle _R$ be an $\Inc(\N)$-invariant ideal in $R$,
where  $a_1,\dots,a_c \in \N$. Set $I_n = I \cap R_n=\langle x_{1,1}^{a_1},\ldots,x_{c,1}^{a_c},\dots, x_{1,n}^{a_1},\ldots,x_{c,n}^{a_c}\rangle$ and $|a| = a_1 + \cdots + a_c$. Then using Lemma \ref{reg mod regular element} it is easy to see that 
\[
\reg (I_n) =  (|a| - c) n + 1\quad \text{for every }\ n \ge 1.
\]
\end{example}


\section{Existence of generic initial chains}\label{sec7}

The existence of generic initial ideals in a Noetherian polynomial ring reduces the study of Hilbert series and many other interesting invariants of an arbitrary graded ideal to the case of a Borel-fixed ideal (or a strongly stable ideal if the base field has characteristic 0). This technique is very powerful and has many interesting applications; see, e.g., \cite{Gr}. Given the theory of equivariant Hilbert series of $\Inc(\N)^i$-invariant chains developed in \cite{NR17}, it is natural to ask whether there is a \emph{good} notion of generic initial chains in this context.

Let $\Icc=(I_n)_{n\ge 1}$ be an $\Inc(\N)^i$-invariant chain of graded ideals. Then a generic initial chain $\Jcc$ of $\Icc$ should have the following properties:

\begin{enumerate}
 \item
 $\Jcc=(J_n)_{n\ge 1}$ is an $\Inc(\N)^i$-invariant chain of monomial ideals,
 \item
 $\Icc$ and $\Jcc$ have the same equivariant Hilbert series, i.e., $I_n$ and $J_n$ have the same Hilbert series as ideals in $R_n$ for every $n\ge 1$,
 \item
 $J_n$ is a Borel-fixed ideal in $R_n$ for every $n\ge 1$.
\end{enumerate}

Given a monomial order $\le$ on $R$, the most natural and reasonable candidate for such a generic initial chain would be the chain $(\gin_{\le}(I_n))_{n\ge 1}$, where $\gin_{\le}(I_n)$ denotes the generic initial ideal of $I_n$ with respect to $\le$ in $R_n$. This chain clearly satisfies Conditions (ii) and (iii) above. Unfortunately, as the next result shows, it does not satisfy (i) in general.

\begin{prop}
 \label{generic}
 Assume the base field $K$ has characteristic $0$. Let $\Icc=(I_n)_{n\ge 1}$ be an $\Inc(\N)^i$-invariant chain of graded ideals. If the chain $(\gin_{\le}(I_n))_{n\ge 1}$ is  $\Inc(\N)^i$-invariant, then the sequence $(\reg(I_n))_{n\ge 1}$ is bounded.
\end{prop}

\begin{proof}
 It is well-known that $\reg(I_n)$ is bounded by the maximal degree $\delta_n$ of a minimal generator of $\gin_{\le}(I_n)$; see, e.g., \cite[Corollary 2.12]{Gr}. If $(\gin_{\le}(I_n))_{n\ge 1}$ is an $\Inc(\N)^i$-invariant chain, then the sequence $(\delta_n)_{n\ge 1}$ is eventually constant, which yields the desired statement.
\end{proof}

In view of Proposition \ref{mon_orbit}, Theorem \ref{thm_extremal}, and Example \ref{artinian}, the previous proposition suggests that there might not be a satisfactory notion of generic initial chains (of $\Inc(\N)^i$-invariant chains) in general. Nevertheless, one can still expect the existence of generic initial chains of special chains, such as those ones defined by squarefree monomials. We conclude the paper with the following conjecture, which might be of interest from the combinatorial point of view.

\begin{conj}
 \label{gen_squarefree}
 Assume $c=1$. Let $\Icc=(I_n)_{n\ge 1}$ be an $\Inc(\N)^i$-invariant chain of squarefree monomial ideals. Then the chain $(\gin_{\rev}(I_n))_{n\ge 1}$ is also $\Inc(\N)^i$-invariant, where $\rev$ denotes the reverse lexicographic order on $R$.
\end{conj}

\end{document}